\documentclass{amsart}

\usepackage{hyperref}
\hypersetup{
    colorlinks=true,
    linkcolor=blue,
    citecolor=blue,
    urlcolor=blue,
}

\makeatletter
\newcommand\org@maketitle{}
\newcommand\@authors{}
\let\org@maketitle\maketitle
\def\maketitle{%
	\let\@authors\authors
	\nxandlist{; }{ and }{; }\@authors
	\hypersetup{
		linktocpage=true,
		pdftitle={\@title},
                pdfauthor={\@authors},
                pdfsubject={\subjclassname. \@subjclass},
		pdfkeywords={\@keywords}
	}%
	\org@maketitle
}
\makeatother

\usepackage{amssymb}
\usepackage{rsfso}
\usepackage{mathtools}
\usepackage{microtype}
\usepackage[alphabetic,msc-links]{amsrefs}

\usepackage{doi}
\renewcommand{\PrintDOI}[1]{\doi{#1}}

\numberwithin{equation}{section}
\newtheorem{maintheorem}{Theorem}

\newtheorem{theorem}{Theorem}[section]
\newtheorem{lemma}[theorem]{Lemma}
\newtheorem{corollary}[theorem]{Corollary}
\newtheorem{proposition}[theorem]{Proposition}
\theoremstyle{definition}
\newtheorem{definition}[theorem]{Definition}
\newtheorem{example}[theorem]{Example}
\theoremstyle{remark}
\newtheorem{remark}[theorem]{Remark}

\newcommand{\al}{\alpha}
\newcommand{\be}{\beta}
\newcommand{\g}{\gamma}

\newcommand{\e}{\varepsilon}

\newcommand{\si}{\sigma}
\newcommand{\om}{\omega}
\newcommand{\La}{\Delta}

\newcommand{\olR} {\overline R}

\newcommand{\R}{\mathbb{R}}
\newcommand{\vp}{\varphi}

\newcommand{\iy}{{\infty}}
\newcommand{\D}{\nabla}
\newcommand{\ra}{\rightarrow}
\newcommand{\pa}{\partial}

\renewcommand{\div}{\operatorname{div}}
\newcommand{\loc}{\mathrm{loc}}
\newcommand{\dist}{\operatorname{dist}}
\newcommand{\mean}[1]{\langle #1\rangle}

\def\Xint#1{\mathchoice
  {\XXint\displaystyle\textstyle{#1}}%
  {\XXint\textstyle\scriptstyle{#1}}%
  {\XXint\scriptstyle\scriptscriptstyle{#1}}%
  {\XXint\scriptscriptstyle\scriptscriptstyle{#1}}%
  \!\int}
\def\XXint#1#2#3{{\setbox0=\hbox{$#1{#2#3}{\int}$}
    \vcenter{\hbox{$#2#3$}}\kern-.5\wd0}}

\def\dashint{\Xint-}

\mathchardef\ordinarycolon\mathcode`\:
\mathcode`\:=\string"8000
\begingroup \catcode`\:=\active
  \gdef:{\mathrel{\mathop\ordinarycolon}}
\endgroup

\author{Seongmin Jeon}
\address{Department of Mathematics, Purdue University, West Lafayette,
  IN 47907, USA}
\email[S.J.]{jeon54@purdue.edu}
\author{Arshak Petrosyan}
\address{Department of Mathematics, Purdue University, West Lafayette,
  IN 47907, USA}
\email[A.P.]{arshak@purdue.edu}
\thanks{The second author is supported in part by NSF Grant DMS-1800527.}

\title[Almost minimizers for certain fractional variational
problems]{Almost minimizers for certain fractional variational problems}


\dedicatory{To Nina Nikolaevna Uraltseva on the occasion of
  her 85th birthday.}

\allowdisplaybreaks

\begin{document}

\begin{abstract} In this paper we introduce a notion of almost
  minimizers for certain variational problems governed by the
  fractional Laplacian, with the help of the Caffarelli-Silvestre
  extension. In particular, we study almost fractional harmonic
  functions and almost minimizers for the fractional obstacle problem
  with zero obstacle. We show that for a certain range of parameters,
  almost minimizers are almost Lipschitz or $C^{1,\beta}$-regular.
\end{abstract}

\maketitle


\section{Introduction and Main Results}


\subsection{Fractional harmonic functions}
Given $0<s<1$, we say that a function
$u\in \mathcal{L}_s(\R^n):=L^1(\R^n,(1+|x|^{n+2s})^{-1})$ is
\emph{$s$-fractional harmonic} in an open set $\Omega\subset\R^n$ if
\begin{equation}\label{eq:u-s-frac-pv}
  (-\Delta_x)^s
  u(x):=C_{n,s}\,\text{p.v.}\int_{\R^n}\frac{u(x)-u(x+z)}{|z|^{n+2s}}=0\quad\text{in
  }\Omega,
\end{equation}
where p.v.\ stands for Cauchy's principal value and $C_{n,s}$ is a
normalization constant.  The formula above is just one of many
equivalent definitions of the fractional Laplacian $(-\Delta_x)^s$,
another one being a pseudo-differential operator with Fourier symbol
$|\xi|^{2s}$. We refer to a recent review of Garofalo \cite{Gar19} for
basic properties of $(-\Delta_x)^s$, as well as many historical
remarks concerning that operator.

In recent years, there has been a surge of interest in nonlocal
problems involving the fractional Laplacian, when it was discovered
that the problems can be localized by the use of the so-called
Caffarelli-Silvestre extension procedure \cite{CafSil07}. Namely, for
$a=1-2s\in(-1,1)$, let
$$
P(x, y):=C_{n,
  a}\frac{|y|^{1-a}}{\left(|x|^2+|y|^2\right)^{\frac{n+1-a}2}},\quad
(x,y)\in\R^n\times\R_+=\R^{n+1}_+,
$$
(to be called the Poisson kernel for the extension operator $L_a$) and
consider the convolution, still denoted by $u$,
$$
u(x,y):=u*P(\cdot,y)=\int_{\R^n}
u(z)P(x-z,y)dz,\quad(x,y)\in\R^{n+1}_+.
$$
Note that $u(x,y)$ solves the Cauchy problem
\begin{align*}
  L_a u:=\div(|y|^a\nabla u)=0&\quad\text{in }\R^{n+1}_+,\\*
  u(x,0)=u(x)&\quad\text{on }\R^n,
\end{align*}
where $\nabla=\nabla_{x,y}$ is the full gradient in $x$ and $y$
variables. $L_a$ is known as the Caffarelli-Silvestre \emph{extension
  operator}. Then, one can recover $(-\Delta_x)^s u$ as the fractional
normal derivative on $\R^n$
$$
(-\Delta_x)^{s}u(x)=-C_{n,a}\lim_{y\to 0+}y^a\partial_y u(x,y),\quad
x\in\R^n
$$
to be understood in the appropriate sense of traces. Now, going back
to the definition \eqref{eq:u-s-frac-pv}, if we consider the even
reflection of $u$ in $y$-variable to all of $\R^{n+1}$, i.e.,
$$
u(x,y)=u(x,-y),\quad x\in\R^n,\ y<0,
$$
then the following fact holds: $u(x)$ is $s$-fractional harmonic in
$\Omega$ if and only if $u(x,y)$ satisfies
\begin{equation}\label{eq:La-harm-ext}
  L_a u=0\quad\text{in }\widetilde \Omega:=\R^{n+1}_-\cup
  (\Omega\times\{0\})\cup\R^{n+1}_+.
\end{equation}
(We will refer to solutions of $L_au=0$ as \emph{$L_a$-harmonic
  functions}.) This is essentially Lemma~4.1 in \cite{CafSil07}. Since
$L_a u=0$ in $\R^n_\pm$ by definition, the condition
\eqref{eq:La-harm-ext} is equivalent to asking
$$
L_a u=0\quad\text{in }B_r(x_0),
$$
for any ball $B_r(x_0)$ centered at $x_0\in \Omega$ such that
$B_r(x_0)\Subset \widetilde\Omega$, or equivalently
$B_r'(x_0)\Subset \Omega$. Now, observing that the solutions of the
above equation are minimizers of the weighted Dirichlet energy
$\int_{B_r(x_0)}|\nabla v|^2|y|^a$, we obtain the following fact.

\begin{proposition} A function $u\in \mathcal{L}_s(\R^n)$ is
  $s$-fractional harmonic in $\Omega$ if and only if its reflected
  Caffarelli-Silvestre extension $u(x,y)$ is in
  $W^{1,2}_{\loc}(\widetilde\Omega,|y|^a)$ and for any ball $B_r(x_0)$
  with $x_0\in\Omega$ such that $B'_r(x_0)\Subset\Omega$, we have
$$
\int_{B_r(x_0)}|\nabla u|^2|y|^a\leq \int_{B_r(x_0)}|\nabla v|^2|y|^a,
$$
for any $v\in u+W^{1,2}_0(B_r(x_0),|y|^a)$.
\end{proposition}

We take this proposition as the starting point for the definition of
almost $s$-fractional harmonic functions, in the spirit of Anzellotti
\cite{Anz83}.

\begin{definition}[Almost $s$-fractional harmonic functions] Let
  $r_0>0$ and $\omega:(0,r_0)\to [0,\infty)$ be a modulus of
  continuity\footnote{i.e., a nondecreasing function with
    $\omega(0+)=0$}. We say that a function $u\in\mathcal{L}_s(\R^n)$
  is \emph{almost $s$-fractional harmonic} in an open set
  $\Omega\subset \R^n$, with a gauge function $\omega$, if its
  reflected Caffarelli-Silvestre extension $u(x,y)$ is in
  $W^{1,2}_{\loc}(\widetilde\Omega,|y|^a)$ and for any ball $B_r(x_0)$
  with $x_0\in\Omega$ and $0<r<r_0$ such that
  $B'_r(x_0)\Subset\Omega$, we have
  \begin{equation}
    \int_{B_r(x_0)} |\nabla u|^2|y|^a\leq
    (1+\omega(r))\int_{B_r(x_0)} |\nabla v|^2 |y|^a,
  \end{equation}
  for any $v\in u+W^{1,2}_0(B_r(x_0),|y|^a)$.
\end{definition}


\subsection{Fractional obstacle problem} A function
$u\in\mathcal{L}_s(\R^n)$ is said to solve the $s$-fractional obstacle
problem with obstacle $\psi$ in an open set $\Omega\subset\R^n$, if
\begin{equation}\label{eq:frac-obs-prob}
  \min\{(-\Delta_x)^s u, u-\psi\}=0\quad\text{in }\Omega.
\end{equation}
We refer to \cites{Sil07,CafSalSil08} for general introduction and
basic results on this problem. With the help of the reflected
Caffarelli-Silvestre extension, we can rewrite the problem as a
Signorini-type problem for the operator $L_a$:
\begin{align*}
  L_a u=0&\quad\text{in }\R^{n+1}_\pm\\
  \min\{-\partial^a_y u, u-\psi\}=0&\quad\text{in }\Omega, 
\end{align*}
where
$$
\partial_y^a u(x,0):=\lim_{y\to 0+}y^a\partial_y u(x,y).
$$
This, in turn, can be written in the following variational form, see
\cite{CafSalSil08}.
\begin{proposition} A function $u\in\mathcal{L}_s(\R^n)$ solves
  \eqref{eq:frac-obs-prob} if and only if its reflected
  Caffarelli-Silvestre extension $u(x,y)$ is in
  $W^{1,2}_\loc(\widetilde\Omega)$ and for any ball $B_r(x_0)$ with
  $x_0\in\Omega$ such that $B'_r(x_0)\Subset\Omega$, we have
$$
\int_{B_r(x_0)}|\nabla u|^2|y|^a\leq \int_{B_r(x_0)}|\nabla v|^2|y|^a,
$$
for any
$v\in\mathfrak{K}_{\psi,u}(B_r(x_0),|y|^a):=\{v\in
u+W^{1,2}_0(B_r,|y|^a): v\geq \psi\ \text{on}\ B_r'(x_0)\}$.
\end{proposition}

\begin{definition}[Almost minimizers for $s$-fractional obstacle
  problem]
  \label{def:alomost-min-frac-obst}
  Let $r_0>0$ and $\omega:(0,r_0)\to [0,\infty)$ be a modulus of
  continuity. We say that a function $u\in \mathcal{L}_s(\R^n)$ is an
  \emph{almost minimizer for the $s$-fractional obstacle problem} in
  an open set $\Omega\subset \R^n$, with a gauge function $\omega$, if
  its reflected Caffarelli-Silvestre extension $u(x,y)$ is in
  $W^{1,2}_{\loc}(\widetilde\Omega,|y|^a)$ and for any ball $B_r(x_0)$
  with $x_0\in\Omega$ and $0<r<r_0$ such that
  $B'_r(x_0)\Subset\Omega$, we have
  \begin{equation}\label{eq:almost-min-frac-obst}
    \int_{B_r(x_0)}|\nabla u|^2|y|^a \leq
    (1+\omega(r))\int_{B_r(x_0)} |\nabla v|^2|y|^a,
  \end{equation}
  for any $v\in\mathfrak{K}_{\psi,u}(B_r(x_0),|y|^a)$.
\end{definition}

The notion of almost minimizers above is related to the one for the
thin obstacle problem ($s=1/2$) studied by the authors in
\cite{JeoPet19a}, but there are certain important differences. In
Definition~\ref{def:alomost-min-frac-obst}, we ask the almost
minimizing property \eqref{eq:almost-min-frac-obst} to hold only for
balls centered on the ``thin space'' $\R^n$, while in
\cite{JeoPet19a}, we ask that property for balls centered at any point
in an open set in the ``thick space'' $\R^{n+1}$. In a sense, this
means that here we think of the perturbation from minimizers as living
on the thin space, while in \cite{JeoPet19a} they live in the thick
space.


\subsection{Main results and structure of the paper}

In this paper, our main concern is the regularity of almost minimizers
in their original variables.

We start with examples of almost minimizers in
Section~\ref{sec:exampl-almost-minim}. We then proceed to prove the
following results, echoing those in \cite{Anz83} and \cite{JeoPet19a}.

\begin{maintheorem}
  \label{mthm:I}
  Let $u\in\mathcal{L}_s(\R^n)$ be almost $s$-fractional harmonic in
  $\Omega$. Then
  \begin{enumerate}
  \item $u$ is almost Lipschitz in $\Omega$, i.e,
    $u\in C^{0,\sigma}(\Omega)$ for any $0<\sigma<1$.
  \item If $\omega(r)=r^\alpha$, then $u\in C^{1,\beta}(\Omega)$ for
    some $\beta=\beta_{n,a,\alpha}>0$.
  \item If $0<s<1/2$ or $s=1/2$ and $\omega(r)=r^\alpha$ for some
    $\alpha>0$, then $u$ is actually $s$-fractional harmonic in
    $\Omega$.
  \end{enumerate}
\end{maintheorem}

In the case of the $s$-fractional obstacle problem, our results are
obtained under the assumption that $1/2\leq s<1$ and $\psi=0$.

\begin{maintheorem}
  \label{mthm:II}
  Let $u\in\mathcal{L}_s(\R^n)$ be an almost minimizer for the
  $s$-fractional obstacle problem with obstacle $\psi=0$ in $\Omega$.
  \begin{enumerate}
  \item If $1/2\leq s<1$, then $u\in C^{0,\sigma}(\Omega)$ for any
    $0<\sigma<1$.
  \item If $1/2\leq s<1$ and $\omega(r)=r^\alpha$ for some $\alpha>0$,
    then $u\in C^{1,\beta}(\Omega)$ for some
    $\beta=\beta_{n,a,\alpha}>0$.
  \end{enumerate}
\end{maintheorem}

The proofs follow the general approach in \cite{Anz83} and
\cite{JeoPet19a} by first obtaining growth estimates for minimizers
(see Section~\ref{sec:growth-estim-minim}) and then deriving their
perturbed versions for almost minimizers
(Section~\ref{sec:almost-s-fractional} for $s$-fractional harmonic
functions and Section~\ref{sec:alm-min-s-frac-obst} for the
$s$-fractional obstacle problem). The regularity then follows by an
embedding theorem of a Morrey-Campanato-type space into the H\"older
space, which we included in
Appendix~\ref{sec:morr-camp-space}. Finally,
Appendix~\ref{sec:La-poly-expansion} contains the proof of orthogonal
polynomial expansion of $L_a$-harmonic functions, that we rely on in
deriving the growth estimates in
Section~\ref{sec:growth-estim-minim}. The polynomial expansion has
other interesting corollaries such as the (known) real-analyticity of
$s$-fractional harmonic functions, which are of independent interest.


\subsection{Notation}
Throughout the paper we use the following notation. $\R^n$ is the
$n$-dimensional Euclidean space. The points of $\R^{n+1}$ are denoted
by $X=(x, y)$, where $x=(x_1,\ldots,x_n)\in \R^{n}$, $y\in\R$. We
routinely identify $x\in \R^{n}$ with $(x, 0)\in \R^{n}\times
\{0\}$. $\R^{n+1}_\pm$ stands for open halfspaces
$\{X=(x,y)\in\R^{n+1}: \pm y>0\}$.

We use the following notations for balls of radius $r$ in $\R^n$ and
$\R^{n+1}$
\begin{alignat*}{2}
  B_r(X)&=\{Z\in \R^{n+1}:|X-Z|<r\},&\quad&\text{(Euclidean) ball in $\R^{n+1}$},\\
  B^{\pm}_r(x)&=B_r(x,0)\cap \{\pm y>0\},&& \text{half-ball in $\R^{n+1}$},\\
  B'_r(x)&=B_r(x,0)\cap \{ y=0\}, &&\text{ball in $\R^{n}$}.
\end{alignat*}
We typically drop the center from the notation if it is the
origin. Thus, $B_r=B_r(0)$, $B'_r=B'_r(0)$, etc.

Next,
$\nabla
u=\nabla_{X}u=(\partial_{x_1}u,\ldots,\partial_{x_n}u,\partial_yu)$
stands for the full gradient, while
$\nabla_x u=(\partial_{x_1}u,\ldots,\partial_{x_n}u)$. We also use the
standard notations for partial derivatives, such as $\partial_{x_i}u$,
$u_{x_i}$, $u_y$ etc.

In integrals, we often drop the variable and the measure of
integration if it is with respect to the Lebesgue measure or the
surface measure. Thus,
$$
\int_{B_r} u|y|^a=\int_{B_r} u(X)|y|^a dX,\quad \int_{\partial B_r}
u|y|^a=\int_{\partial B_r} u(X)|y|^a dS_X,
$$
where $S_X$ stands for the surface measure.

By $L^2(B_R,|y|^a)$ and $L^2(\partial B_R,|y|^a)$ we indicate the
weighted Lebesgue spaces of functions with the norms
\begin{align*}
  \|u\|_{L^2(B_R,|y|^a)}^2&=\int_{B_R}u^2 |y|^a\\
  \|u\|_{L^2(\partial B_R,|y|^a)}^2&=\int_{\partial B_R}u^2 |y|^a.
\end{align*}
$W^{1,2}(B_R,|y|^a)$ is the corresponding weighted Sobolev space of
functions with the norm
$$
\|u\|_{W^{1,2}(B_R,|y|^a)}^2=\|u\|_{L^2(B_R,|y|^a)}^2+\|\D
u\|_{L^2(B_R,|y|^a)}^2.
$$
We also use other typical notations for Sobolev spaces. Thus,
$W^{1,2}_0(B_R,|y|^a)$ stands for the closure of $C^\infty_0(B_R)$ in
$W^{1,2}(B_R,|y|^a)$.

For $x\in\R^n$ and $r>0$, we indicate by $\mean{u}_{x, r}$ the
$|y|^a$-weighted integral mean value of a function $u$ over
$B_r(x)$. That is,
$$
\mean{u}_{x, r}=\dashint_{B_r(x)} u|y|^a
=\frac{1}{\om_{n+1+a}r^{n+1+a}}\int_{B_r(x)}u |y|^a,
$$ where $\omega_{n+1+a}=\int_{B_1}|y|^a$ is the $|y|^a$-weighted
volume of the unit ball $B_1$ in $\R^{n+1}$. Similarly to the other
notations, we drop the origin if it is $0$ and write $\mean{u}_r$ for
$\mean{u}_{0,r}$.


\section{Examples of almost minimizers}
\label{sec:exampl-almost-minim}

Before we proceed with the proofs of the main results, we would like
to give some examples of almost minimizers.

\begin{example}\label{example-anz} Let $u\in\mathcal{L}_s(\R^n)$ be a
  solution of
  $$
  (-\La_x)^su+b(x)\cdot \D_x u=0\quad\text{in }\Omega,
  $$
  where $b=(b^1, b^2, \ldots, b^n)\in W^{1, \infty}(\Omega)$ and
  $1/2<s<1$ (or $-1<a<0$). Then $u$ is an almost $s$-fractional
  harmonic with a gauge function $\omega(r)=C r^{-a}$ (note that
  $-a>0$).
\end{example}

\begin{proof} Consider a ball $B_r(x_0)$ centered at $x_0\in \Omega$
  such that $B'_r(x_0)\Subset \Omega$. Without loss of generality
  assume that $x_0=0$. Let $v$ be the minimizer of
$$
\int_{B_r}|\D v|^2 |y|^a
$$
on $u+W^{1,2}(B_r,|y|^a)$.  Then
$$
\int_{B_r}\nabla v\nabla(u-v)|y|^a=0,
$$
and as a consequence,
\begin{align*}
  \int_{B_r}(|\D u|^2-|\D v|^2)|y|^a
  &= \int_{B_r}|\D(u-v)|^2|y|^a.
\end{align*}
Then, we have
\begin{align*}
  \int_{B_r}(|\D u|^2-|\D v|^2)|y|^a&= 2\int_{B_r^+}|\D(u-v)|^2|y|^a\\*
                                    &= 2\int_{B_r^+}|\D(u-v)|^2|y|^a+\div(|y|^a\D(u-v))\,(u-v)\\
                                    &= 2\int_{B_r^+}\div\left(|y|^a\D\left(\frac{(u-v)^2}2\right)\right)\\
                                    &= 2\int_{(\pa B_r)^+}|y|^a(u-v)(u_{\nu}-v_{\nu})-2\int_{B_r'}(u-v)(\partial_y^a u-\partial_y^a v)\\
                                    &= C\int_{B_r'}(u-v)(-\La_x)^su\\*
                                    &=-C\int_{B_r'}(u-v)b^i u_{x_i}  
\end{align*}
with $C=C_{n,a}$.  Next, extending $b^i$ to $\R^{n+1}$ by
$b^i(x, y):=b^i(x)$, we have
\begin{align*}
  \int_{B_r}(|\D u|^2-|\D v|^2)|y|^a
  &= -C\int_{B_r'}(u-v)b^iu_{x_i}\\
  &= C\int_{B_r^+}\pa_{y}\left( (u-v)b^iu_{x_i}\right)\\
  &= C\int_{B_r^+}(u_y-v_y)b^iu_{x_i}+(u-v)b^iu_{x_iy}\\
  &\le C\|b\|_{W^{1, \infty}(\Omega)}\int_{B_r^+}|\D u|^2+|\D v|^2\\
  &\qquad+C\int_{\pa (B_r^+)}(u-v)b^iu_y\nu_{x_i}-C\int_{B_r^+}\pa_{x_i}((u-v)b^i)u_y\\
  &=  C\|b\|_{W^{1, \infty}(\Omega)}\int_{B_r^+}|\D u|^2+|\D v|^2\\
  &\qquad -C\int_{B_r^+}((u_{x_i}-v_{x_i})b^i+(u-v)b^i_{x_i})u_y\\
  &\le C\|b\|_{W^{1, \infty}(\Omega)}\int_{B_r^+}|\D u|^2+|\D v|^2+|u-v|^2.
\end{align*}
Using Poincare's inequality, it follows that
\begin{align*}
  \int_{B_r}|y|^a(|\D u|^2-|\D v|^2) &\le C\int_{B_r}|\D u|^2+|\D v|^2\\
                                     &\le Cr^{-a}\int_{B_r}(|\D u|^2+|\D v|^2)|y|^a\\
                                     &\le C r^{-a}\int_{B_r}|\D u|^2|y|^a.
\end{align*}
Hence,
$$
\int_{B_r(x_0)}|\nabla u|^2|y|^a\leq (1+C
r^{-a})\int_{B_r(x_0)}|\nabla v|^2|y|^a,
$$
for $0<r<r_0$, with $C$ and $r_0$ depending on $n$, $a$, and
$\|b\|_{W^{1,\infty}(\Omega)}$.
\end{proof}

\begin{example}\label{example-sig}
  Let $u\in\mathcal{L}_s(\R^n)$ be a solution of the obstacle problem
  for fractional Laplacian with drift
  \begin{align*}
    \min\{(-\Delta_x)^s u+b(x)\cdot\nabla_x u,u\}=0\quad\text{in }\Omega,
  \end{align*}
  where $b=(b^1, b^2, \ldots, b^n)\in W^{1, \infty}(\Omega)$ and
  $1/2<s<1$ (or $-1<a<0$). Then $u$ is an almost minimizer for
  $s$-fractional obstacle problem in $\Omega$ with an obstacle
  $\psi=0$ and a gauge function $\omega(r)=C r^{-a}$.
\end{example}

The obstacle problem above has been studied earlier in \cite{PetPop15}
and \cite{GarPetPopSVG17}.

\begin{proof} We argue similarly to Example~\ref{example-anz}. Let
  $B_r(x_0)$ centered at $x_0\in \Omega$ such that
  $B'_r(x_0)\Subset \Omega$. Without loss of generality assume that
  $x_0=0$. Let $v$ be the minimizer of
  $$
  \int_{B_r}|\nabla v|^2|y|^a
  $$
  on
  $\mathfrak{K}_{0,u}(B_r,|y|^a)=\{v\in u+W^{1,2}_0(B_r,|y|^a): v\geq
  0\ \text{on}\ B_r'(x_0)\}$.  Next, we write
  \begin{align*}
    \int_{B_r}(|\D u|^2-|\D v|^2)|y|^a
    &= 2\int_{B_r}\nabla u\nabla(u-v)|y|^a-\int_{B_r}|\nabla
      (u-v)|^2|y|^a\\
    &\leq 2\int_{B_r}\nabla u\nabla(u-v)|y|^a\\
    &=4\int_{B_r^+}\nabla u\nabla(u-v)|y|^a+\div(|y|^a\nabla u)(u-v)\\
    &=-4\int_{B_r'}(u-v)\partial_y^au\\
    &=C\int_{B_r'}(u-v)(-\Delta_x)^su\\
    &\begin{multlined}  
      = C\biggl[-\int_{B_r'\cap \{u>0\}}(u-v)b^iu_{x_i}+\int_{B_r'\cap
        \{u=0\}}(-v)\,(-\La_x)^su \biggr]
    \end{multlined}\\
    &\begin{multlined} \le
       C\biggl[-\int_{B_r'\cap\{u>0\}}(u-v)b^iu_{x_i}-\int_{B_r'\cap
       \{u=0\}}(-v)b^iu_{x_i}\biggr]
   \end{multlined}\\
    &= -C\int_{B_r'}(u-v)b^iu_{x_i},
  \end{align*}
  where we used that $(-\La)^su+b^iu_{x_i}\ge0$ and $-v\le 0$ on
  $B_r'\cap \{u=0\}$ in the last inequality.

  Then we complete the proof as in Example~\ref{example-anz}.
\end{proof}


\section{Growth estimates for minimizers}
\label{sec:growth-estim-minim}

In this section we prove growth estimates for $L_a$-harmonic functions
and solutions of the Signorini problem for $L_a$, i.e., minimizers of
$v$ of the weighted Dirichlet integral
$$
\int_{B_r}|\nabla v|^2|y|^a
$$
on $v+W^{1,2}_0(B_r,|y|^a)$ or on the thin obstacle constraint set
$\mathfrak{K}_{0,v}(B_r,|y|^a)$.

The idea is that these estimates will extend to almost minimizers and
will ultimately imply their regularity with the help of
Morrey-Campanato-type space embedding.

The proofs in this section are akin to those in \cite{JeoPet19a} for
almost minimizers of the thin obstacle problem. Yet, one has to be
careful with different growth rates for tangential and normal
derivatives.


\subsection{Growth estimates for $L_a$-harmonic functions}

\begin{lemma}\label{anz-even v-mvp}
  Let $v\in W^{1, 2}(B_R, |y|^a)$ be a solution of $L_av=0$ in
  $B_R$. If $v$ is even in $y$, then for $0<\rho<R$
  \begin{align*}
    \int_{B_{\rho}}|\D_xv|^2|y|^a
    &\le \left(\frac{\rho}R\right)^{n+1+a}\int_{B_R}|\D_xv|^2|y|^a\\
    \int_{B_{\rho}}|v_y|^2|y|^a
    &\le\left(\frac{\rho}R\right)^{n+3+a}\int_{B_R}|v_y|^2|y|^a.
  \end{align*}
\end{lemma}

\begin{proof}
  Note that we can write $$v(x, y)=\sum_{k=0}^{\infty}p_k(x, y),$$
  where $p_k$'s are $L_a$-harmonic homogeneous polynomials of degree
  $k$ (see Appendix~\ref{sec:La-poly-expansion}). Then
  $\{\pa_{x_i}p_k\}_{k=1}^{\infty}$ are $L_a$-harmonic homogeneous
  polynomials of degree $k-1$, and thus orthogonal in
  $L^2(\pa B_1, |y|^a)$. Thus,
  \begin{align*}
    \int_{B_{\rho}}|\D_xv|^2|y|^a &= \sum_{k=1}^{\infty}\int_{B_{\rho}}|\D_xp_k|^2|y|^a\\
                                  &= \sum_{k=1}^{\infty}\left(\frac{\rho}R\right)^{n+1+a+2(k-1)}\int_{B_R}|\D_xp_k|^2|y|^a\\
                                  &\le \left(\frac{\rho}R\right)^{n+1+a} \sum_{k=1}^{\infty}\int_{B_R}|\D_xp_k|^2|y|^a\\
                                  &=\left(\frac{\rho}R\right)^{n+1+a}\int_{B_R}|\D_xv|^2|y|^a.
  \end{align*}
  Similarly, $\{|y|^a\pa_yp_k\}_{k=1}^{\infty}$ are $L_{-a}$-harmonic
  homogeneous functions of degree $k-1+a$, and thus orthogonal in
  $L^2(\pa B_1, |y|^{-a})$. Notice that since $p_1(x, y)=p_1(x)$ is
  independent of $y$ variable by the even symmetry, we have
  $|y|^a\pa_yp_1=0$. Thus,
  \begin{align*}
    \int_{B_{\rho}}|v_y|^2|y|^a &= \int_{B_{\rho}}\left| |y|^av_y\right|^2|y|^{-a}\\
                                &= \sum_{k=2}^{\infty}\int_{B_{\rho}}\left| |y|^a\pa_yp_k\right|^2|y|^{-a}\\
                                &=
                                  \sum_{k=2}^{\infty}\left(\frac{\rho}R\right)^{n+1-a+2(k-1+a)}\int_{B_R}| |y|^a\pa_yp_k|^2|y|^{-a}\\
                                &\le \left(\frac{\rho}R\right)^{n+3+a}\int_{B_R}|v_y|^2|y|^a.\qedhere
  \end{align*}
\end{proof}

\begin{lemma} \label{anz-grad-v-holder} Let $v$ be a solution of
  $L_a v=0$ in $B_R$, even in $y$. Then, for $0<\rho<R$,
  \begin{equation}
    \int_{B_{\rho}}|\D_xv-\mean{\D_xv}_{\rho}|^2|y|^a\le \left(\frac{\rho}R\right)^{n+a+3}\int_{B_R}|\D_xv-\mean{\D_xv}_R|^2|y|^a.
  \end{equation}
\end{lemma}
\begin{proof} First note that since $L_a(\D_xv)=0$ in $B_R$,
  $\mean{\nabla_x v}=\nabla_x v(0)$ by the mean value theorem for
  $L_a$-harmonic functions, see \cite{CafSalSil08}*{Lemma~2.9}. If we
  use the expansion $v=\sum_{k=0}^\infty p_k(x,y)$ in $B_R$ as in the
  proof of Lemma~\ref{anz-even v-mvp}, then
  $\D_xv-\D_x v(0)=\sum_{k=2}^\infty \D_x p_k$ and consequently
  \begin{align*}
    \int_{B_\rho}|\D_xv-\D_x v(0)|^2|y|^a&=\sum_{k=2}^\infty
                                           \int_{B_\rho}|\D_x p_k|^2|y|^a\\
                                         &=\sum_{k=2}^\infty
                                           \left(\frac{\rho}{R}\right)^{n+a+2k-1}\int_{B_R}|\D_x
                                           p_k|^2|y|^a\\
                                         &\leq \left(\frac{\rho}{R}\right)^{n+a+3}\sum_{k=2}^\infty
                                           \int_{B_R}|\D_x p_k|^2|y|^a\\*
                                         &= \left(\frac{\rho}{R}\right)^{n+a+3}\int_{B_R}|\D_xv-\D_x v(0)|^2|y|^a.\qedhere
  \end{align*}
\end{proof}


\subsection{Growth estimates for the solutions of the Signorini
  problem for $L_a$}

Our estimates for the solutions of the Signorini problem will require
an assumption that $1/2\leq s<1$, or $a\leq 0$. Also, unless stated
otherwise, the obstacle $\psi$ is assumed to be zero.

The first estimate is the analogue of Lemma~\ref{anz-even v-mvp}, but
with less information of the growth of $v_y$.
\begin{lemma}
  Let $v$ be a solution of the Signorini problem for $L_a$ in $B_R$,
  even in $y$, with $a\leq 0$. Then, for $0<\rho<R$
  \begin{equation}\label{sig-v-mon}
    \int_{B_{\rho}}|\D v|^2|y|^a\le \left(\frac{\rho}R \right)^{n+1+a}\int_{B_R}|\D v|^2|y|^a.
  \end{equation}
\end{lemma}
\begin{proof} We use the following property: if $v$ is as in the
  statement of the lemma, then $v_{x_i}$, $i=1,\ldots,n$, and
  $y|y|^{a-1}v_y$ are H\"older continuous in $B_R$, see
  \cite{CafSalSil08}. Moreover, we have that
$$
L_a (v_{x_i}^\pm)\geq 0,\quad L_{-a} ((y|y|^{a-1}v_y)^\pm)\geq
0\quad\text{in }B_R.
$$
This follows from the fact that $L_a v_{x_i}=0$ in $\{\pm v_{x_i}>0\}$
and $L_{-a} (y|y|^{a-1}v_y)=0$ in $\{\pm y|y|^{a-1}v_y>0\}$, by the
complementarity condition $v_yv=0$ on $B_R'$, as well as an argument
in Exercise 2.6 or Exercise 9.5 in \cite{PetShaUra12}. As a
consequence, we have
  $$
  L_a(|\nabla_x v|^2)\geq 0,\quad L_{-a} (||y|^a v_y|^2)\geq
  0\quad\text{in }B_R.
  $$
  We next use the following $|y|^a$-weighted sub-mean value property
  for $L_a$-subharmonic functions: If $L_a w\geq 0$ weakly in $B_R$,
  $-1<a<1$, then
$$
\rho\mapsto\frac{1}{\rho^{n+1+a}}\int_{B_\rho} w|y|^a
$$
is nondecreasing. This follows by integration from the spherical
sub-mean value property, see \cite{CafSalSil08}*{Lemma~2.9}.  Thus, we
have that
\begin{align*}
  \rho&\mapsto \frac{1}{\rho^{n+1+a}}\int_{B_\rho}|\nabla_x v|^2|y|^a
  \\\rho&\mapsto \frac{1}{\rho^{n+1-a}}\int_{B_\rho}|y|^a u_y^2
\end{align*}
are monotone nondecreasing for $0<\rho<R$. This implies
\begin{align*}
  \int_{B_\rho}|\nabla_x v|^2|y|^a&\leq \left(\frac{\rho}R\right)^{n+1+a}\int_{B_R}|\nabla_x v|^2|y|^a\\
  \int_{B_\rho}v_y^2|y|^a&\leq \left(\frac{\rho}R\right)^{n+1-a}\int_{B_R}v_y^2|y|^a.
\end{align*}
In the case $a\leq 0$, we therefore conclude that the
bound~\eqref{sig-v-mon} holds.
\end{proof}

\begin{lemma}\label{sig-grad-v-est1}
  Let $v$ be a solution of the Signorini problem for $L_a$ in $B_R$,
  even in $y$, with $a\leq 0$. If $v(0)=0$, then there exists
  $C=C_{n, \al}$ such that for $0<\rho<r<(3/4)R$,
  \begin{multline*}
    \int_{B_{\rho}}|\D_x v-\mean{\D_x v}_{\rho}|^2|y|^a \le
    \left(\frac{\rho}r\right)^{n+a+3}\int_{B_r}|\D_x v-\mean{\D_x
      v}_{r}|^2|y|^a\\+C\|v\|_{L^{\infty}(B_R)}^2\frac{\rho^{n+2}}{R^{2+2s}}
  \end{multline*}
\end{lemma}

\begin{proof}
  Define
  $$\vp(r):=\frac 1{r^{n+a+3}}\int_{B_r}|\D_x v-\mean{\D
    _xv}_r|^2|y|^a.$$ Then,
  \begin{align*}
    \vp(r) &= \frac 1{r^{n+a+3}}\left[\int_{B_r}|\D_x v|^2|y|^a-2\mean{\D _xv}_r\int_{B_r}\D_x v|y|^a+\mean{\D_x v}_r^2\int_{B_r}|y|^a\right]\\
           &= \frac 1{r^{n+a+3}}\left[\int_{B_r}|\D _xv|^2|y|^a-\frac 1{\om_{n+1+a}r^{n+1+a}}\left(\int_{B_r}\D_x v|y|^a\right)^2\right].
  \end{align*}
  Thus, using the Cauchy-Schwarz and Young's inequality, we obtain
  \begin{align*}
    \vp'(r)
    &= \frac1{r^{n+a+3}} \bigg[-\frac{n+a+3}r\int_{B_r}|\D_x
      v|^2|y|^a +\int_{\pa B_r}|\D_x v|^2|y|^a\\
    &\qquad +\frac{n+a+3}{\om_{n+1+a}r^{n+2+a}}\left(\int_{B_r}\D
      _xv|y|^a\right)^2
      +\frac{n+1+a}{\om_{n+1+a}r^{n+2+a}}\left(\int_{B_r}\D
      _xv|y|^a\right)^2  \\
    &\qquad -\frac 2{\om_{n+1+a}r^{n+1+a}}\left(\int_{B_r}\D_x v|y|^a\right)\left(\int_{\pa B_r}\D_x v|y|^a\right)\bigg]\\
    &\ge -\frac C{r^{n+a+3}}\left[\frac 1r\int_{B_r}|\D _xv|^2|y|^a+\left(\frac 1r\int_{B_r}|\D_x v|^2|y|^a\right)^{1/2}\left(\int_{\pa B_r}|\D_x v|^2|y|^a\right)^{1/2}\right]\\
    &\ge  -\frac C{r^{n+a+3}}\left[\frac 1r\int_{B_r}|\D_x v|^2|y|^a+\int_{\pa B_r}|\D_x v|^2|y|^a\right].
  \end{align*}
  Next, we note that
$$
[\D_xv]_{C^{0, s}\left(B_{3/4R}\right)}\le \frac
{C_{n,s}}{R^{1+s}}\|v\|_{L^{\infty}(B_R)}.
$$
Indeed, this follows from the known interior regularity for solutions
of the Signorini problem for $L_a$ in $B_1$ in the case $R=1$, see
e.g. \cite{CafSalSil08}, and a simple scaling argument for all
$R>0$. Noting also that $\D_xv(0)=0$, since $v$ attains its minimum on
$B'_r$ at $0$, we have that for $X\in \overline{B_r}$ with $r<(3/4)R$
$$
|\D_xv(X)|=|\D_xv(X)-\D _xv(0)|\le \frac
C{R^{1+s}}\|v\|_{L^{\infty}(B_R)}r^s
$$
and so
$$\frac 1r\int_{B_r}|\D_x v|^2|y|^a+\int_{\pa B_r}|\D_x v|^2|y|^a\le
C\|v\|^2_{L^{\infty}(B_R)}\frac{r^{n+1}}{R^{2+2s}}.$$ This
gives
$$\vp'(r)\ge -\frac C{r^{a+2}}\|v\|_{L^{\infty}(B_R)}^2\frac
1{R^{2+2s}}.$$ Thus, for $0<\rho<r<(3/4)R$,
\begin{align*}
  \vp(r)-\vp(\rho) &= \int_{\rho}^r\vp'(t)\,dt\\*
                   &\ge -C\|v\|_{L^{\infty}(B_R)}^2\frac {\rho^{-1-a}-r^{-1-a}}{R^{2+2s}}.
\end{align*}
Therefore,
\begin{align*}
  &\int_{B_{\rho}}|\D_x v-\mean{\D_x v}_{\rho}|^2|y|^a \\*
  &\qquad\qquad= \rho^{n+a+3}\vp(\rho)\\
  &\qquad\qquad\le \rho^{n+a+3}\left(\vp(r)+C\|v\|_{L^{\infty}(B_R)}^2\frac {\rho^{-1-a}-r^{-1-a}}{R^{2+2s}}\right)\\
  &\qquad\qquad\le \left(\frac{\rho}r\right)^{n+a+3}\int_{B_r}|\D_x v-\mean{\D_x v}_r|^2|y|^a+C\|v\|^2_{L^{\iy}(B_R)}\frac{\rho^{n+2}}{R^{2+2s}}.
    \qedhere
\end{align*}
\end{proof}

\begin{lemma}\label{sig-grad-v-est2}
  Let $v$ be a solution of the Signorini problem for $L_a$ in $B_R$,
  even in $y$. Then there are $C_1=C_{n, a}$, $C_2=C_{n, a}$ such that
  for all $0<\rho< S<(3/8)R,$
  \begin{multline*}
    \int_{B_{\rho}}|\D_xv-\mean{\D _xv}_{\rho}|^2|y|^a \le
    C_1\Big(\frac{\rho}{S}\Big)^{n+a+3}\int_{B_S}|\D_xv-\mean{\D_xv}_{S}|^2|y|^a\\*+C_2
    \|v\|_{L^{\infty}(B_R)}^2\frac{S^{n+2}}{R^{2+2s}}.
  \end{multline*}
\end{lemma}

\begin{proof}If $\rho\ge S/8$, then we immediately have
  \begin{align*}
    \int_{B_{\rho}}|\D_xv-\mean{\D_xv}_{\rho}|^2|y|^a &\le C\left(\frac{8\rho}{S}\right)^{n+a+3}\int_{B_{\rho}}|\D_xv-\mean{\D_xv}_\rho|^2|y|^a\\
                                                      &\le C\left(\frac{\rho}{S}\right)^{n+a+3}\int_{B_S}|\D_xv-\mean{\D_xv}_S|^2|y|^a.
  \end{align*}
  Thus we may assume $\rho<S/8$. Due to Lemma~\ref{sig-grad-v-est1},
  we may assume $v(0)>0$. Let
  $d:=\dist\left(0, \{v(\cdot, 0)=0\}\right)>0$. Then $L_av=0$ in
  $B_d$. Thus, if $d\ge S$, we may use Lemma~\ref{anz-grad-v-holder}
  to obtain
  \begin{align*}
    \int_{B_{\rho}}|\D_xv-\mean{\D_xv}_{\rho}|^2|y|^a &\le \left(\frac{\rho}S\right)^{n+a+3}\int_{B_{S}}|\D_xv-\mean{\D_xv}_{S}|^2|y|^a.
  \end{align*}
  Thus we may also assume $d<S$.

  \medskip\noindent \emph{Case 1.} $S/4\le d\,(<S)$.

  \medskip\noindent \emph{Case 1.1.} Suppose $0<\rho<d\,(<S)$. Then
  using $L_a(\D_xv)=0$ in $B_d$ again,
  \begin{align*}
    \int_{B_{\rho}}|\D_xv-\mean{\D_xv}_{\rho}|^2|y|^a&\le \left(\frac{\rho}d\right)^{n+a+3}\int_{B_{d}}|\D_xv-\mean{\D_xv}_{d}|^2|y|^a\\
                                                     &\le C\left(\frac{\rho}S\right)^{n+a+3}\int_{B_S}|\D_xv-\mean{\D_xv}_S|^2|y|^a.
  \end{align*}

  \medskip\noindent \emph{Case 1.2.} Suppose $\rho\ge d\,(\ge
  S/4)$. Then
$$
\int_{B_{\rho}}|\D_xv-\mean{\D_xv}_{\rho}|^2|y|^a\le
\left(\frac{4\rho}S\right)^{n+a+3}\int_{B_S}|\D_xv-\mean{\D_xv}_S|^2|y|^a.$$

\medskip\noindent \emph{Case 2.} $0<d<S/4$.

\medskip\noindent \emph{Case 2.1.} Suppose $\rho<d/2$. Take
$x_1\in \pa(B_d')$ such that $v(x_1)=0$. Then using inclusions
$B_{\rho}\subset B_{d/2}\subset B_{(3/2)d}(x_1)\subset
B_{S/2}(x_1)\subset B_{R/2}(x_1)$, $L_av=0$ in $B_d$ and the preceding
Lemma~\ref{sig-grad-v-est1}, we obtain
\begin{align*}
  &\int_{B_{\rho}}|\D_xv-\mean{\D_xv}_{\rho}|^2|y|^a\\*
  &\qquad\begin{aligned}
    &\le \left(\frac{2\rho}d\right)^{n+a+3}\int_{B_{d/2}}|\D_xv-\mean{\D_xv}_{d/2}|^2|y|^a\\
    &\le \left(\frac{2\rho}d\right)^{n+a+3}\int_{B_{(3/2)d}(x_1)}|\D_xv-\mean{\D_xv}_{x_1, (3/2)d}|^2|y|^a\\
    &\begin{multlined}
      \le
      \left(\frac{2\rho}d\right)^{n+a+3} \biggl[\left(\frac{3d}S\right)^{n+a+3}\int_{B_{S/2}(x_1)}|\D_xv-\mean{\D_xv}_{x_1,S/2}|^s|y|^a\\+C\|v\|^2_{L^{\infty}(B_{R/2}(x_1))}\frac{S^{n+2}}{R^{2+2s}}\biggr]
    \end{multlined}
    \\
    &\le
    C\left(\frac{\rho}S\right)^{n+a+3}\int_{B_S}|\D_xv-\mean{\D_xv}_S|^2|y|^a+C\|v\|^2_{L^{\infty}(B_R)}\frac{S^{n+2}}{R^{2+2s}}
  \end{aligned}
\end{align*}

\medskip\noindent \emph{Case 2.2.} Suppose $d/2\le \rho$. Then we see
that $B_{\rho}\subset B_{3\rho}(x_1)\subset B_{S/2}(x_1)\subset
B_S$. As we did in Case 2.1, we have
\begin{align*}
  &\int_{B_{\rho}}|\D_xv-\mean{\D_xv}_{\rho}|^2|y|^a\\
  &\qquad\le \int_{B_{3\rho}(x_1)}|\D_xv-\mean{\D_xv}_{x_1, 3\rho}|^2|y|^a\\
  &\qquad
    \begin{multlined}
      \le
      C\left(\frac{\rho}S\right)^{n+a+3}\int_{B_{S/2}(x_1)}|\D_xv-\mean{\D_xv}_{x_1,S/2}|^2|y|^a\\
      +C\|v\|^2_{L^{\iy}(B_{R/2}(x_1))}\frac{S^{n+2}}{R^{2+2s}}
    \end{multlined}
  \\
  &\qquad\le C\left(\frac{\rho}S\right)^{n+a+3}\int_{B_S}|\D_xv-\mean{\D_xv}_S|^2|y|^a+C\|v\|^2_{L^{\iy}(B_R)}\frac{S^{n+2}}{R^{2+2s}}.\qedhere
\end{align*}
\end{proof}

\begin{corollary}\label{sig-grad-v-est3}
  Let $v$ be a solution of the Signorini problem for $L_a$ in $B_R$,
  even in $y$. Then there are $C_1=C_{n, a}$, $C_2=C_{n, a}$ such that
  for all $0<\rho< S<(3/16)R,$
  \begin{multline*}
    \int_{B_{\rho}}|\D_xv-\mean{\D_xv}_{\rho}|^2|y|^a \le
    C_1\Big(\frac{\rho}{S}\Big)^{n+a+3}\int_{B_S}|\D_xv-\mean{\D_xv}_{S}|^2|y|^2\\+C_2
    \mean{v^2}_{R}\frac{S^{n+2}}{R^{2+2s}}.
  \end{multline*}
\end{corollary}

\begin{proof}
  Since $v^{\pm}=\max(\pm v, 0)\ge 0$ and $L_a(v^{\pm})=0$ in
  $\{v^{\pm}>0\}$, we have $L_a(v^{\pm})\ge 0$ in $B_R$. (For this,
  one may follow the argument in Exercise 2.6 or Exercise 9.5 in
  \cite{PetShaUra12}.) Thus, we have by Theorem~2.3.1 in
  \cite{FabKenSer82}
  $$ \sup_{B_{R/2}}v^{\pm}\le C\left(\frac
    1{\om_{n+1+a}R^{n+1+a}}\int_{B_R}\left(v^{\pm}\right)^2|y|^a\right)^{1/2}.
$$
Hence, $$\|v\|_{L^{\iy}(B_{R/2})}^2 \le C\mean{v^2}_{R},$$ which
completes the proof.
\end{proof}


\section{Almost $s$-fractional harmonic functions}
\label{sec:almost-s-fractional}

In this section we prove Theorem~\ref{mthm:I}, by deducing growth
estimates for almost minimizers from that of minimizers and then
applying the Morrey-Campanato space embedding to deduce the regularity
of almost minimizers.

\begin{theorem}[Almost Lipschitz regularity]
  \label{thm:anz-holder}
  If $u$ is an almost $s$-fractional harmonic function in $B_1'$,
  $0<s<1$, then $u\in C^{0, \si}(B'_1)$ for any $0<\si<1$.
\end{theorem}

Besides the growth estimates for minimizers we will also need the
following lemma.

\begin{lemma}\label{lem:HL}
  Let $r_0>0$ be a positive number and let
  $\vp:(0,r_0)\to (0, \infty)$ be a nondecreasing function. Let $a$,
  $\beta$, and $\gamma$ be such that $a>0$, $\gamma >\beta >0$. There
  exist two positive numbers $\e=\e_{a, \gamma,\beta}$,
  $c=c_{a,\gamma,\beta}$ such that, if
$$
\vp(\rho)\le
a\Bigl[\Bigl(\frac{\rho}{r}\Bigr)^{\gamma}+\e\Bigr]\vp(r)+b\, r^{\be}
$$ for all $\rho$, $r$ with $0<\rho\leq r<r_0$, where $b\ge 0$,
then one also has, still for $0<\rho<r<r_0$,
$$
\vp(\rho)\le
c\Bigl[\Bigl(\frac{\rho}{r}\Bigr)^{\be}\vp(r)+b\rho^{\be}\Bigr].
$$
\end{lemma}
\begin{proof} See Lemma~3.4 in \cite{HanLin97}.
\end{proof}

\begin{proof}[Proof of Theorem~\ref{thm:anz-holder}]
  Let $K$ be a compact subset of $B_1'$ containing $0$. Take
  $\delta=\delta_{n, \omega, \si, K}>0$ such that
  $\delta<\dist(K, \pa B_1')$ and $\omega(\delta)\le \e$, where
  $\e=\e_{2, n+1+a, n-1+a+2\si}$ is as Lemma~\ref{lem:HL}. For
  $0<R<\delta$, let $v$ be a minimizer of
$$
\int_{B_R}|\D v|^2 |y|^a
$$ 
on $u+W^{1,2}_0(B_R)$. Then $L_av=0$ in $B_R$. In particular,
\begin{align*}\int_{B_R}|y|^a\D v\cdot \D(u-v)
  =0,
\end{align*}
and hence
\begin{align*}
  \int_{B_R}|\D(u-v)|^2|y|^a&= \int_{B_R}|\D u|^2|y|^a-\int_{B_R}|\D v|^2|y|^a-2\int_{B_R}|y|^a\D v\cdot \D(u-v)\\
                            &\le \omega(R)\int_{B_R}|\D v|^2|y|^a.
\end{align*}
Moreover, by Lemma~\ref{anz-even v-mvp}, for $0<\rho<R$ we have
$$ \int_{B_{\rho}}|\D v|^2|y|^a\le
\left(\frac{\rho}R\right)^{n+1+a}\int_{B_R}|\D v|^2|y|^a.
$$
Thus \begin{align*}
       \int_{B_{\rho}}|\D u|^2|y|^a &\le 2\int_{B_{\rho}}|\D v|^2|y|^a+2\int_{B_{\rho}}|\D(u-v)|^2|y|^a\\
                                    &\le 2\left(\frac{\rho}R\right)^{n+1+a}\int_{B_R}|\D v|^2|y|^a+2\int_{B_{\rho}}|\D(u-v)|^2|y|^a\\
                                    &\le 2\left(\frac{\rho}R\right)^{n+1+a}\int_{B_R}|\D v|^2|y|^a+2\omega(R)\int_{B_R}|\D v|^2|y|^a\\*
                                    &\le
                                      2\left[\left(\frac{\rho}R\right)^{n+1+a}+\e\right]\int_{B_R}|\D
                                      u|^2|y|^a.
     \end{align*}
     By Lemma~\ref{lem:HL},
     \begin{align*}
       \int_{B_{\rho}}|\D u|^2|y|^a
       &\le C_{n, a, \si}\left(\frac{\rho}R\right)^{n-1+a+2\si}\int_{B_R}|\D u|^2|y|^a,
     \end{align*}
     for any $0<\sigma<1$.  Taking $R\nearrow \delta$ we have
     \begin{equation}\label{anz-holder-1}\int_{B_{\rho}}|\D
       u|^2|y|^a\le
       C_{n, a, \si,\delta}\|\D u\|^2_{L^{2}(B_1, |y|^a)}\rho^{n-1+a+2\si}.
     \end{equation}
     By weighted Poincar\'e inequality
     \cite{FabKenSer82}*{Theorem~(1.5)}
$$\int_{B_{\rho}}|u-\mean{u}_{\rho}|^2|y|^a\le C_{n, a, \si,\delta}\|\D u\|^2_{L^{2}(B_1, |y|^a)}\rho^{n+1+a+2\si}.$$
Now, a similar estimates holds at all point $x_0\in K$, which implies
the H\"older continuity of $u$ (see Theorem~\ref{thm:morr-camp-space})
with
\[
  \|u\|_{C^{0, \si}(K)} \le C_{n, a,\omega, \si, K}\|u\|_{W^{1,2}(B_1,
    |y|^a)}.\qedhere
\]
\end{proof}

\begin{theorem}[$C^{1,\beta}$ regularity]

  \label{anz-grad-u-holder} If $u$ is an almost $s$-fractional
  harmonic function in $B_1'$, $0<s<1$, with gauge function
  $\omega(r)=r^\alpha$, $\alpha>0$, then $\D_xu\in C^{0,\be}(B'_1)$
  for some $\beta=\beta(n,s,\alpha)$.
\end{theorem}

\begin{proof}
  Let $K\Subset B_1'$ be a ball and take
  $0<\delta<\dist(K, \pa B_1')$. Let $B_R'(x_0)\Subset B_1'$ with
  $0<R<\delta$, for $x_0\in K$. For simplicity write $x_0=0$, and let
  $v$ be the $L_a$-harmonic function in $B_R$ with $v=u$ on $\pa B_R$.
  Then, by Jensen's inequality we have
$$
\int_{B_{\rho}}|\mean{\D_xu}_{\rho}-\mean{\D_xv}_{\rho}|^2|y|^a \le
\int_{B_{\rho}}|\D_xu-\D_xv|^2|y|^a,
$$
and hence
\begin{align*}
  \int_{B_{\rho}}|\D_xu-\mean{\D_xu}_{\rho}|^2|y|^a
  &\leq
    3\int_{B_{\rho}}|\D_xv-\mean{\D_xv}_{\rho}|^2|y|^a+3\int_{B_{\rho}}|\D_xu-\D_xv|^2|y|^a\\
  &\qquad +3\int_{B_{\rho}}|\mean{\D_xu}_{\rho}-\mean{\D_xv}_{\rho}|^2|y|^a\\
  &\leq 3\int_{B_{\rho}}|\D_xv-\mean{\D_xv}_{\rho}|^2|y|^a+6\int_{B_{\rho}}|\D_xu-\D_xv|^2|y|^a.
\end{align*}
Similarly,
$$
\int_{B_R}|\D_xv-\mean{\D_xv}_R|^2|y|^a \le
3\int_{B_{R}}|\D_xu-\mean{\D_xu}_R|^2|y|^a+6\int_{B_{R}}|\D_xu-\D_xv|^2|y|^a.
$$
Next let $\be\in(0,\alpha/2)$. Then using the estimate
\eqref{anz-holder-1} in the proof of Theorem~\ref{thm:anz-holder} with
$\si=1+\be-\frac{\al}2$, we have
\begin{align*}
  \int_{B_R}|\D u-\D v|^2|y|^a
  &= \int_{B_R}|\D u|^2|y|^a-\int_{B_R}|\D v|^2|y|^a \\ 
  &\le R^{\al}\int_{B_R}|\D u|^2|y|^a\\
  &\le C\|\D u\|^2_{L^2(B_1, |y|^a)}R^{n+1+a+2\be}.
\end{align*}
Then, with the help of Lemma~\ref{anz-grad-v-holder}, we have that for
$\rho< R$
\begin{align*}
  &\int_{B_{\rho}}|\D_xu-\mean{\D_xu}_{\rho}|^2|y|^a\\
  &\qquad \le C\int_{B_{\rho}}|\D_xv-\mean{\D_xv}_{\rho}|^2|y|^a+C\int_{B_{\rho}}|\D_xu-\D_xv|^2|y|^a\\
  &\qquad\le C\left(\frac{\rho}R\right)^{n+a+3}\int_{B_R}|\D_xv-\mean{\D_xv}_R|^2|y|^a+C\int_{B_{\rho}}|\D_xu-\D_xv|^2|y|^a\\
  &\qquad\le C\left(\frac{\rho}R\right)^{n+a+3}\int_{B_R}|\D_xu-\mean{\D_xu}_R|^2|y|^a+C\int_{B_R}|\D_xu-\D_xv|^2|y|^a\\
  &\qquad \le C\left(\frac{\rho}R\right)^{n+a+3}\int_{B_R}|\D_xu-\mean{\D_xu}_R|^2|y|^a+C\|\D u\|^2_{L^2(B_1, |y|^a)}R^{n+1+a+2\be}.
\end{align*}
Hence, by Lemma~\ref{lem:HL}, we obtain that for $\rho< R$
\begin{multline*}
  \int_{B_{\rho}}|\D_xu-\mean{\D_xu}_{\rho}|^2|y|^a\\
  \le
  C\left[\left(\frac{\rho}R\right)^{n+1+a+2\be}\int_{B_R}|\D_xu-\mean{\D_xu}_R|^2|y|^a+\|\D
    u\|^2_{L^2(B_1, |y|^a)}\rho^{n+1+a+2\be}\right].
\end{multline*}
Taking $R\nearrow \delta$, we have
\begin{align*}
  \int_{B_{\rho}}|\D_xu-\mean{\D_xu}_{\rho}|^2|y|^a &\le C_{n, a, \al, \be, K}\|\D u\|^2_{L^2(B_1, |y|^a)}\rho^{n+1+a+2\be}.
\end{align*}
Now, a similar estimate holds for any $x_0\in K$. Fixing $\beta$ and
applying Theorem~\ref{thm:morr-camp-space}, we have
\[\|\D_xu\|_{C^{0, \be}(K)}\le C_{n, a, \al, K}\|u\|_{W^{1,2}(B_1,
    |y|^a)}.\qedhere
\]
\end{proof}

\begin{remark}\label{rem:trace-weak}
  From the assumption for almost minimizers that the
  Caffarelli-Silvestre extension $u\in W^{1,2}_\loc$ we know only that
  $\nabla_x u\in L^2_\loc$, which is not sufficient to deduce the
  existence of the trace of $\nabla_x u$ on $B_1'$. However, in the
  proof of Theorem~\ref{anz-grad-u-holder} we showed that $\nabla_x u$
  is in a Morrey-Campanato space, which implies the existence of the
  trace as the limit of averages
$$
T(\D_xu)(x_0)=\lim_{r\ra 0+}\mean{\nabla_x u}_{x_0,r}.
$$
It is not hard to see that $T(\D_x u)$ is the distributional
derivative $\nabla_x u$ on $B_1'$. Indeed, if
$\eta\in C^\infty_0(B_1')$, then extending it to $\R^{n+1}$ by
$\eta(x,y)=\eta(x)$, we have
\begin{align*}
  \int_{B_1'}T(\partial_{x_i} u)\eta
  &=\lim_{r\to 0+}\int_{B_1'}\mean{\partial_{x_i}u}_{x,r}\eta=\lim_{r\to 0+}\int_{B_1'}\partial_{x_i}u \mean{\eta}_{x,r}\\
  &=\lim_{r\to 0+}-\int_{B_1'}u \mean{\partial_{x_i}\eta}_{x,r}=-\int_{B_1'} u\partial_{x_i}\eta.
\end{align*}
\end{remark}

\begin{theorem}\label{anz-u_y-holder} Let $u$ be an almost
  $s$-fractional harmonic function in $B_1'$ for $0<s<1/2$ or $s=1/2$
  and a gauge function $\omega(r)=r^\alpha$ for some $\alpha>0$. Then
  $u$ is actually $s$-fractional harmonic in $B_1'$.
\end{theorem}

\begin{proof} We argue as in the proof
  Theorem~\ref{thm:anz-holder}. Let $K$, $\delta$, $R$, $v$ be as in
  the proof of that theorem. Then, by Lemma~\ref{anz-even v-mvp}, for
  $0<\rho<R$
$$
\int_{B_{\rho}}|v_y|^2|y|^a\le
\left(\frac{\rho}R\right)^{n+3+a}\int_{B_R}|v_y|^2|y|^a.
$$
Thus, for any $0<\si<1$, we have
\begin{align*}
  \int_{B_{\rho}}| |y|^au_y|^2|y|^{-a}
  &\le 2\int_{B_{\rho}}|v_y|^2|y|^a+2\int_{B_{\rho}}|u_y-v_y|^2|y|^a\\
  &\le 2\left(\frac{\rho}R\right)^{n+3+a}\int_{B_R}|v_y|^2|y|^a+2\int_{B_{\rho}}|u_y-v_y|^2|y|^a\\
  &\le 4\left(\frac{\rho}R\right)^{n+3+a}\int_{B_R}|u_y|^2|y|^a+6\int_{B_R}|u_y-v_y|^2|y|^a\\
  &\le 4\left(\frac{\rho}R\right)^{n+3+a}\int_{B_R}| |y|^au_y|^2|y|^{-a}+6\omega(R)\int_{B_R}|\D u|^2|y|^a\\
  &\le 4\left(\frac{\rho}R\right)^{n+3+a}\int_{B_R}|
    |y|^au_y|^2|y|^{-a}\\
  &\qquad +C_{n, a, \si, \delta}\omega(R)\|\D u\|^2_{L^2(B_1, |y|^a)}R^{n-1+a+2\si},
\end{align*}
where we used \eqref{anz-holder-1} in the last inequality.

Consider now the two cases in statement of the theorem.

\medskip\noindent \emph{Case 1}. $0<s<1/2$ (or $a>0$). In this case by
Lemma~\ref{lem:HL},
\begin{align*}
  &\int_{B_{\rho}}| |y|^au_y|^2|y|^{-a}\\
  &\qquad\le C\left[\left(\frac{\rho}R\right)^{n-1+a+2\si}\int_{B_R}| |y|^au_y|^2|y|^{-a}+\omega(\delta)\|\D u\|^2_{L^2(B_1, |y|^a)}\rho^{n-1+a+2\si}\right]\\
  &\qquad\le C\|\D u\|^2_{L^2(B_1, |y|^a)}\rho^{n+1-a+(-2+2a+2\si)}. 
\end{align*}
Now we take $\si=1-a/2\in (0, 1)$ to have $-2+2a+2\si=a>0$. Varying
the center, we have a similar bound at every $x\in K$. Then, by
Theorem~\ref{thm:morr-camp-space}, we obtain that the limit of the
averages $T(y|y|^{a-1} u_y)=0$ on $B_1'$. This implies that
$(-\Delta_x)^su=0$ on $B_1'$. Indeed, arguing as in
Remark~\ref{rem:trace-weak}, by considering the mollifications $u_\e$
in $x$-variable, we note that
$$
\int_{B_{\rho}}| |y|^a(u_\e)_y|^2|y|^{-a}\leq C \rho^{n+1-a+a}
$$
which implies that $T(y|y|^{a-1}(u_\e)_y)=0$ on $K\Subset B_1'$. On
the other hand, $u_\e\in C^2\cap\mathcal{L}_s(\R^n)$, which implies
that $y|y|^{a-1}(u_\e)_y$ is continuous up to $y=0$, since we can
explicitly write, for $y>0$, the symmetrized formula
$$
y^a(u_\e)_y(x,y)=\int_{\R^n}\frac{u_\e(x+z)+u_\e(x-z)-2u_\e(x)}{|z|^2}|z|^2y^a\partial_y
P(z,y)dz
$$
with locally integrable kernel
$ |z|^2|y^a\partial_y P(z,y)|\leq C/|z|^{n-1-a}.  $ Hence, we obtain
that $(-\Delta_x)^su_\e=\partial_y^a u_\e=0$ on the ball
$K\Subset B_1'$. Then, passing to the limit as $\e\to 0$, this implies
that $(-\Delta_x)^su=0$ in $B_1'$.

\medskip\noindent \emph{Case 2.} $s=1/2$ (or $a=0$) and
$\omega(r)=r^\alpha$.  In this case, we have a bound
\begin{align*}
  \int_{B_{\rho}}|u_y|^2
  \le 4\left(\frac{\rho}R\right)^{n+3}\int_{B_R}|u_y|^2+C\|\D u\|^2_{L^2(B_1)}R^{n-1+2\si+\alpha},
\end{align*}
Them, by Lemma~\ref{lem:HL}, we have
\begin{align*}
  \int_{B_{\rho}}|u_y|^2&\le C\left[\left(\frac{\rho}R\right)^{n-1+2\si+\alpha}\int_{B_R}|u_y|^2+\|\D u\|^2_{L^2(B_1)}\rho^{n-1+2\si+\alpha}\right]\\*
                        &\le C\|\D u\|^2_{L^2(B_1)}\rho^{n+1+(\al-2+2\si)}. 
\end{align*}
Taking $1-\alpha/4<\sigma<1$, we can guarantee that
$\al-2+2\si>\alpha/2>0$, which implies that $T(y|y|^{-1}u_y)=0$ on
$B_1'$. Then, arguing as at the end of Case 1, we conclude that
$(-\Delta_x)^{1/2}u=0$ in $B_{1}'$.
\end{proof}

We finish this section with formal proof of Theorem~\ref{mthm:I}.

\begin{proof}[Proof of Theorem~\ref{mthm:I}] Parts (1), (2), and (3)
  are proved in Theorems~\ref{thm:anz-holder},
  \ref{anz-grad-u-holder}, and \ref{anz-u_y-holder}, respectively.
\end{proof}


\section{Almost minimizers for $s$-fractional obstacle problem}
\label{sec:alm-min-s-frac-obst}

In this section we investigate the regularity of almost minimizers for
the $s$-fractional obstacle problem with zero obstacle and give a
proof of Theorem~\ref{mthm:II}. All results in this section are proved
under the assumption $1/2\leq s<1$, or $-1<a\leq 0$.

\begin{theorem}[Almost Lipschitz regularity]

  \label{thm:sig-u-holder}
  Let $u$ be an almost minimizer for $s$-fractional obstacle problem
  with zero obstacle in $B_1'$, for $1/2\leq s<1$.  Then
  $u\in C^{0, \si}(B'_1)$ for any $0<\si<1$ with
$$
\|u\|_{C^{0, \si}(K)} \le C_{n, a,\omega, \si, K}\|u\|_{W^{1,2}(B_1,
  |y|^a)}
$$ for any $K\Subset B_1'$.
\end{theorem}

\begin{proof}
  Let $K\Subset B_1'$ with $0\in K$. Take
  $\delta=\delta_{n, a, \omega, \si, K}>0$ such that
  $\delta<\dist(K, \pa B_1')$ and $\omega(\delta)\le \e$, where
  $\e=\e_{2, n+1+a, n-1+a+2\si}$ as in Lemma~\ref{lem:HL}. For
  $0<R<\delta$, let $v$ be the minimizer of
$$
\int_{B_R}|\D v|^2|y|^a
$$
on $\mathfrak{K}_{0,u}(B_R,|y|^a)$. Then $v$ satisfies the variational
inequality
$$
\int_{B_R}\D v\D(w-v)|y|^a\geq 0
$$
for any $w\in \mathfrak{K}_{0,u}(B_R,|y|^a)$. Particularly, taking
$w=u$, we have
$$
\int_{B_R}\D v\D(u-v)|y|^a\geq 0.
$$
As a consequence,
\begin{align*}
  \int_{B_R}|\D(u-v)|^2|y|^a&= \int_{B_R}|\D u|^2|y|^a-\int_{B_R}|\D v|^2|y|^a+2\int_{B_R}|y|^a\D v\cdot \D(v-u)\\
                            &\le \omega(R)\int_{B_R}|\D v|^2|y|^a.
\end{align*}
Next, we use \eqref{sig-v-mon} to derive a similar estimate for
$u$. We have,
\begin{align*}
  \int_{B_{\rho}}|\D u|^2|y|^a &\le 2\int_{B_{\rho}}|\D v|^2|y|^a+2\int_{B_{\rho}}|\D(u-v)|^2|y|^a\\
                               &\le 2\left(\frac{\rho}R\right)^{n+1+a}\int_{B_R}|\D v|^2|y|^a+2\omega(R)\int_{B_R}|\D v|^2|y|^a\\
                               &\le 2\left[\left(\frac{\rho}R\right)^{n+1+a}+\e\right]\int_{B_R}|\D u|^2|y|^a.
\end{align*}
Hence, by Lemma~\ref{lem:HL},
\begin{align*}
  \int_{B_{\rho}}|\D u|^2|y|^a
  &\le C_{n, a, \si}\left(\frac{\rho}R\right)^{n-1+a+2\si}\int_{B_R}|\D u|^2|y|^a.
\end{align*}
As we have seen in Theorem~\ref{thm:anz-holder}, this implies
\begin{align}\label{sig-u-holder-1}
  \int_{B_{\rho}}|\D u|^2|y|^a &\le C_{n, a, \si, \delta}\|\D u\|^2_{L^2(B_1, |y|^a)}\rho^{n-1+a+2\si}\\
  \intertext{then}\notag
  \int_{B_{\rho}}|u-\mean{u}_{\rho}|^2|y|^a &\le C_{n, a, \si,
                                              \delta}\|\D u\|^2_{L^2(B_1, |y|^a)}\rho^{n+1+a+2\si}
\end{align}
and ultimately
\[
  \|u\|_{C^{0, \si}(K)} \le C_{n, a,\omega, \si, K}\|u\|_{W^{1,2}(B_1,
    |y|^a)}.\qedhere
\]
\end{proof}

\begin{theorem}[$C^{1,\beta}$ regularity]
  \label{sig-grad-u-holder}
  Let $u$ be an almost minimizer for the $s$-fractional obstacle
  problem with zero obstacle in $B_1'$, $1/2\leq s<1$, and a gauge
  function $\omega(r)=r^\alpha$. Then $\D_xu\in C^{0, \be}(B_1')$ for
  $\be<\frac{\al s}{8(n+1+a+\al/2)}$ and for any $K\Subset B_1'$ there
  holds
$$
\|\D_xu\|_{C^{0, \be}(K)} \le C_{n, a, \al, \be, K}\|u\|_{W^{1,
    2}(B_1, |y|^a)}.
$$
\end{theorem}

\begin{proof}
  Let $K$ be a thin ball centered at $0$ such that $K\Subset B_1$. Let
  $\e:=\frac{\al}{4(n+1+a+\al/2)}$ and $\g:=1-\frac{s\e}{2(1-\e)}$. We
  fix $R_0=R_0(n, a, \al, K)>0$ small so that $R_0^{1-\e}\le d/2$,
  where $d:=\dist(K, \pa B_1')$ and
  $R_0<\left(\frac 3{16}\right)^{1/{\e}}$.  Then
  $\widetilde{K}:=\{x\in B_1': \dist(x, K)\le R_0^{1-\e}\}\Subset
  B_1$.  We claim that for $x_0\in K$ and $0<\rho<R<R_0$,
  \begin{equation}\label{eq:sig-holder-want-to-prove}
    \begin{multlined}
      \int_{B_{\rho}(x_0)}|\D_xu-\mean{\D_xu}_{x_0, \rho}|^2|y|^a\\*
      \qquad\le C_{n,
        a}\left(\frac{\rho}R\right)^{n+a+3}\int_{B_R(x_0)}|\D_xu-\mean{\D_xu}_{x_0,
        R}|^2|y|^a\\+C_{n, a, \al, K}\|u\|_{W^{1, 2}(B_1,
        |y|^a)}^2R^{n+1+a+s\e}.
    \end{multlined}
  \end{equation}
  Note that once we have this bound, the proof will follow by the
  application of Lemma~\ref{lem:HL} and
  Theorem~\ref{thm:morr-camp-space}.

  For simplicity we may assume $x_0=0$, and fix $0<R<R_0$. Let
  $\olR:=R^{1-\e}.$ Let $v$ be the minimizer of
  $$ \int_{B_{\olR}}|\D v|^2|y|^a$$ on
  $\mathfrak{K}_{0,u}(B_{\olR},|y|^a)$. Then by (\ref{sig-v-mon}) and
  (\ref{sig-u-holder-1}) with $\si=\g$, for $0<\rho\le
  \olR$ \begin{equation}\label{sig-grad-u-holder-1}\begin{aligned}
      \int_{B_{\rho}}|\D v|^2|y|^a &\le \left(\frac{\rho}{\olR}\right)^{n+1+a}\int_{B_{\olR}}|\D v|^2|y|^a\\
      &\le \left(\frac{\rho}{\olR}\right)^{n+1+a} \int_{B_{\olR}}|\D u|^2|y|^a\\
      &\le C_{n, a, \al, K}\left(\frac{\rho}{\olR}\right)^{n+1+a}\|u\|_{W^{1, 2}(B_1, |y|^a)}^2\olR^{n-1+a+2\g}\\
      &\le C_{n, a, \al, K}\|u\|_{W^{1, 2}(B_1,
        |y|^a)}^2\rho^{n-1+a+2\g}.\end{aligned}
  \end{equation}
  This gives
  \begin{equation}\label{sig-grad-u-holder-2}
    \dashint_{B_{\rho}}|v-v_{\rho}|^2|y|^a \le C_1\|u\|_{W^{1, 2}(B_1,
      |y|^a)}^2\rho^{2\g},\quad C_1=C_{n, a, \al, K}.
  \end{equation}
  Since this estimate holds for any $0<\rho<\olR$, the standard dyadic
  argument gives \begin{equation}\label{sig-grad-u-holder-3}
    |v(0)-\mean{v}_{\olR}|\le C_2\|u\|_{W^{1, 2}(B_1,
      |y|^a)}\olR^{\g},\quad C_2=C_{n, a, \al, K}.
  \end{equation}
  Moreover, using (\ref{sig-v-mon}) and (\ref{sig-u-holder-1}) again,
  we have for any $x_1\in B'_{\olR/2}$, $0<\rho<\olR/2$,
  \begin{equation}\label{sig-grad-u-holder-4}\begin{aligned}
      \int_{B_{\rho}(x_1)}|\D v|^2|y|^a &\le \left(\frac{2\rho}{\olR}\right)^{n+1+a}\int_{B_{\olR/2}(x_1)}|\D v|^2|y|^a\\
      &\le \left(\frac{2\rho}{\olR}\right)^{n+1+a}\int_{B_{\olR}}|\D u|^2|y|^a\\
      &\le C_{n, a, \al, K}\|u\|_{W^{1, 2}(B_1,
        |y|^a)}^2\rho^{n-1+a+2\g},\end{aligned}
  \end{equation}
  which implies
  \begin{equation}\label{sig-grad-u-holder-5} [v]_{C^{0,
        \g}(\,\overline{B'_{R/2}}\,)} \le C_3\|u\|_{W^{1, 2}(B_1,
      |y|^a)},\quad C_3=C_{n, a, \al, K}.
  \end{equation}

  Now we define $$ C_4:=C_1+C_2^2+C_3^2.
$$
Our analysis then distinguishes the following two cases
$$
\mean{v^2}_{\olR}\le 6\, C_4\|u\|_{W^{1, 2}(B_1,
  |y|^a)}^2\olR^{2\g}\qquad\text{or}\qquad \mean{v^2}_{\olR}>
6\,C_4\|u\|_{W^{1, 2}(B_1, |y|^a)}^2\olR^{2\g}.
$$

\medskip\noindent \emph{Case 1.}  Suppose first that
$$\mean{v^2}_{\olR}\le 6C_4\|u\|_{W^{1,
    2}(B_1, |y|^a)}^2\olR^{2\g}.
$$
Note that $R_0<\left(\frac 3{16}\right)^{1/{\e}}$ implies
$R<\frac 3{16}\olR$. Then, using Corollary~\ref{sig-grad-v-est3}, we
see that for $0<\rho<R$,
\begin{align*}
  \int_{B_{\rho}}|\D_xu-\mean{\D_xu}_{\rho}|^2|y|^a&\le 3\int_{B_{\rho}}|\D_xv-\mean{\D_xv}_{\rho}|^2|y|^a+6\int_{B_{\rho}}|\D_xu-\D_xv|^2|y|^a\,dx\\
                                                   &\le C_{n,a}\left(\frac{\rho}R\right)^{n+a+3}\int_{B_R}|\D_xv-\mean{\D_xv}_R|^2|y|^a\\
                                                   &\qquad+C_{n, a}\mean{v^2}_{\olR}\frac{R^{n+2}}{\olR^{2+2s}}+6\int_{B_{\rho}}|\D_xu-\D_xv|^2|y|^a\\
                                                   &\le
                                                     C\left(\frac{\rho}R\right)^{n+a+3} \int_{B_R}|\D_xu-\mean{\D_xu}_R|^2|y|^a\\
                                                   &\qquad+C\mean{v^2}_{\olR}\frac{R^{n+2}}{\olR^{2+2s}}+C\int_{B_R}|\D_xu-\D_xv|^2|y|^a.
\end{align*}
Note that for $\si:=1-\al/4$ \begin{align*}
                               \int_{B_R}|\D_xu-\D_xv|^2|y|^a &\le \int_{B_{\olR}}|\D_xu-\D_xv|^2|y|^a\\
                                                              &\le \olR^{\al}\int_{B_{\olR}}|\D v|^2|y|^a\\
                                                              &\le \olR^{\al}\int_{B_{\olR}}|\D u|^2|y|^a\\
                                                              &\le C_{n, a, \al, K}\olR^{\al}\|u\|_{W^{1, 2}(B_1, |y|^a)}^2\olR^{n-1+a+2\si}\\
                                                              &=
                                                                C\|u\|_{W^{1,
                                                                2}(B_1,
                                                                |y|^a)}^2R^{n+1+a+\al/4}.
                             \end{align*}
                             Moreover by the assumption \begin{align*}
                                                          C\mean{v^2}_{\olR}\frac{R^{n+2}}{\olR^{2+2s}} &\le C_{n, a, \al, K}\|u\|_{W^{1, 2}(B_1, |y|^a)}^2R^{n+2}\olR^{2\g-2-2s}\\
                                                                                                        &=
                                                                                                          C\|u\|_{W^{1,
                                                                                                          2}(B_1,
                                                                                                          |y|^a)}^2R^{n+1+a+s\e}.
                                                        \end{align*}
                                                        Hence, we
                                                        obtain
                                                        \eqref{eq:sig-holder-want-to-prove}
                                                        in this case.

                                                        \medskip\noindent
                                                        \emph{Case
                                                          2}. Now we
                                                        assume
$$
\mean{v^2}_{\olR}> 6\,C_4\|u\|_{W^{1, 2}(B_1, |y|^a)}^2\olR^{2\g}.
$$
Then, by \eqref{sig-grad-u-holder-2} and \eqref{sig-grad-u-holder-3}
we obtain \begin{align*}
            \dashint_{B_{\olR}}|v-v(0)|^2|y|^a &\le 2\dashint_{B_{\olR}}|v-v_{\olR}|^2|y|^a+2\dashint_{B_{\olR}}|v_{\olR}-v(0)|^2|y|^a\\
                                               &\le2C_4\|u\|_{W^{1,
                                                 2}(B_1,
                                                 |y|^a)}^2\olR^{2\g}.
          \end{align*}
          Combining the latter bound and the assumption,
          \begin{align*}
            v(0)^2&= \dashint_{B_{\olR}}|v(0)|^2|y|^a\\*
                  &\ge 1/2\dashint_{B_{\olR}}|v(X)|^2|y|^a-\dashint_{B_{\olR}}|v(X)-v(0)|^2|y|^a\\
                  &\ge C_4\|u\|_{W^{1, 2}(B_1, |y|^a)}^2\olR^{2\g}.
          \end{align*}
          Since $C_4\ge C_3^2$, we have $v> 0$ on $B'_{\olR/2}$ by
          \eqref{sig-grad-u-holder-5}. Thus, $L_av=0$ in $B_{\olR/2}$,
          and by Lemma~\ref{anz-grad-v-holder} we have for
          $0<\rho<R$
          $$ \int_{B_{\rho}}|\D_xv-\mean{\D_xv}_{\rho}|^2|y|^a\le
          \left(\frac{\rho}R\right)^{n+a+3}\int_{B_R}|\D_xv-\mean{\D_xv}_R|^2|y|^a.
$$
Thus,
\begin{align*}
  &\int_{B_{\rho}}|\D_xu-\mean{\D_xu}_{\rho}|^2|y|^a\\*
  &\qquad\le 3\int_{B_{\rho}}|\D_xv-\mean{\D_xv}_{\rho}|^2|y|^a+6\int_{B_{\rho}}|\D_xu-\D_xv|^2|y|^a\\
  &\qquad\le 3\left(\frac{\rho}R\right)^{n+a+3}\int_{B_R}|\D_xv-\mean{\D_xv}_R|^2|y|^a+6\int_{B_{\rho}}|\D_xu-\D_xv|^2|y|^a\\
  &\qquad\le C\left(\frac{\rho}R\right)^{n+a+3}\int_{B_R}|\D_xu-\mean{\D_xu}_R|^2|y|^a+C\int_{B_R}|\D_xu-\D_xv|^2|y|^a\\
  &\qquad\le C\left(\frac{\rho}R\right)^{n+a+3}\int_{B_R}|\D_xu-\mean{\D_xu}_R|^2|y|^a+C\|u\|_{W^{1, 2}(B_1, |y|^a)}^2R^{n+1+a+\al/4}.
\end{align*}
This implies \eqref{eq:sig-holder-want-to-prove} and completes the
proof.
\end{proof}

\begin{proof}[Proof of Theorem~\ref{mthm:II}] Parts (1) and (2) are
  contained in Theorems~\ref{thm:sig-u-holder} and
  \ref{sig-grad-u-holder}, respectively.
\end{proof}

\appendix


\section{Morrey-Campanato-type Space}
\label{sec:morr-camp-space}

\begin{theorem}\label{thm:morr-camp-space}
  Let $u\in L^2(B_1, |y|^a)$ and $M$ be such that
  $\|u\|_{L^2(B_1, |y|^a)}\leq M$ and for some $\si\in (0,1)$
$$
\int_{B_r(x)}|u-\mean{u}_{x, r}|^2|y|^a \le M^2r^{n+1+a+2\si},\quad
\mean{u}_{x, r}=\frac{1}{\om_{n+1+a}r^{n+1+a}}\int_{B_r(x)}u\,|y|^a
$$
for any ball $B_r(x)$ centered at $x=(x, 0)\in B_{1/2}'$ and radius
$0<r<r_0\leq 1/2$. Then for any $x\in B_{1/2}'$ there exists the limit
of averages
$$
Tu(x):=\lim_{r\to 0}\mean{u}_{x,r},
$$
which will also satisfy
$$
\int_{B_{r}(x)}|u-Tu(x)|^2|y|^a\leq C_{n, a, \si} M^2 r^{n+1+a+2\si}.
$$
Moreover, $Tu\in C^{0,\si}(B_{1/2}')$ with
$$
\|Tu\|_{C^{0,\si}(B_{1/2}')}\leq C_{n,a,\si,r_0}M.
$$
\end{theorem}

\begin{remark} Note, we can redefine $u(x,0)=Tu(x)$ for any
  $x\in B_{1/2}'$, making $(x,0)$ a Lebesgue point for $u$.
\end{remark}

\begin{proof}
  Let $x, z\in B_{1/2}'$ and $0<\rho<r<r_0$ be such that
  $B_{\rho}(x)\subset B_{r}(z)$.  Then
  \begin{align*}
    |\mean{u}_{x,\rho}-\mean{u}_{z,r}| &\leq\dashint_{B_\rho(x)}|u-\mean{u}_{z, r}||y|^a \\
                                       &\le \left(\frac r\rho\right)^{n+1+a}\dashint_{B_r(z)}|u-\mean{u}_{z, r}| |y|^a\\
                                       &\le \left(\frac r\rho\right)^{n+1+a} \left(\dashint_{B_r(z)}|u-\mean{u}_{z, r}|^2|y|^a\right)^{1/2}\left(\dashint_{B_r(z)}|y|^a\right)^{1/2}\\
                                       &\le C_{n, a}\left(\frac r\rho\right)^{n+1+a}Mr^{\si}.
  \end{align*}
  Now, taking $x=z$ and using a dyadic argument, we can conclude that
$$
|\mean{u}_{x,\rho}-\mean{u}_{x,r}|\leq C_{n,a, \si}M
r^{\si},\quad\text{for any }0<s=\rho<r<r_0.
$$
Indeed, let $k=0,1,2,\ldots$ be such that
$r/2^{k+1}\leq \rho<r/2^{k}$.  Then
\begin{align*}
  |\mean{u}_{x,\rho}-\mean{u}_{x,r}|
  &\leq \sum_{j=1}^k |\mean{u}_{x,r/2^{j-1}}-\mean{u}_{x,r/2^j}|+|\mean{u}_{x,r/2^k}-\mean{u}_{x,\rho}|\\
  &\leq C_{n, a} M \sum_{j=1}^{k+1} (r/2^{j-1})^\si\leq
    C_{n, a,\si}M r^\si.
\end{align*}
This implies that the limit
$$
Tu(x)=\lim_{r\to 0} \mean{u}_{x,r}
$$
exists and
$$
|Tu(x)-\mean{u}_{x,r}|\leq C_{n,a,\si}M r^\si.
$$
Hence, we also have the H\"older integral bound
$$
\int_{B_{r}(x)}|u-Tu(x)|^2|y|^a\leq C_{n,a,\si} M^2 r^{n+1+a+2\si}.
$$
Besides, we have
$$
|Tu(x)|\leq \mean{u}_{x,r_0}+C_{n,a,\si} M r_0^\si\leq
C_{n,a,\si,r_0}M.
$$
It remains to estimate the H\"older seminorm of $Tu$ on
$B_{1/2}'$. Let $x,z\in B_{1/2}'$ and consider two cases.

\medskip\noindent \emph{Case 1.} If $|x-z|<r_0/4$, let
$r=2|x-z|$. Then note that $B_{r/2}(x)\subset B_{r}(z)$ and therefore
we can write
\begin{align*}
  |Tu(x)-Tu(z)|&\leq |Tu(x)-\mean{u}_{x,r/2}|+|Tu(z)-\mean{u}_{z,r}|+|\mean{u}_{x,r/2}-\mean{u}_{z,r}|\\
               &\leq C_{n,a, \si}M r^\si=C_{n,a,\si} M|x-z|^\si.
\end{align*}

\medskip\noindent \emph{Case 2.} If $|x-z|\geq r_0/4$, then
\begin{align*}
  |Tu(x)-Tu(z)|
  &\leq |Tu(x)|+|Tu(z)|\\
  &\leq C_{n,a, \si, r_0} M\\
  &\leq C_{n,a, \si, r_0}M|x-z|^\si. 
\end{align*}
Thus, we conclude
\[
  \|Tu\|_{C^{0,\si}(B_{1/2}')}\leq C_{n,a, \si,r_0}M.\qedhere
\]
\end{proof}


\section{Polynomial expansion for Caffarelli-Silvestre extension}
\label{sec:La-poly-expansion}

Some of the results in Section~\ref{sec:growth-estim-minim} rely on
polynomial expansion theorem for $L_a$-harmonic functions given below.

\begin{theorem}
  \label{thm:even-a-harm} Let $u\in W^{1,2}(B_1,|y|^a)$, $-1<a<1$, be
  a weak solution of the equation $L_au=0$ in $B_1$, even in $y$. Then
  we have the following polynomial expansion:
  $$
  u(x,y)=\sum_{k=0}^\infty p_k(x,y)\quad\text{}
  $$
  locally uniformly in $B_1$, where $p_k(x,y)$ are $L_a$-harmonic
  polynomials, homogeneous of degree $k$ and even in $y$. Moreover,
  the polynomials $p_k$ above are orthogonal in
  $L^2(\partial B_1,|y|^a)$, i.e.,
$$
\int_{\partial B_1}p_k p_m |y|^a=0,\quad k\neq m.
$$
In, particular, $u$ is real analytic in $B_1$.
\end{theorem}

This theorem has the following immediate corollaries, which are of
independent interest and are likely known in the literature. We state
them here for reader's convenience and for possible future reference.

\begin{corollary}\label{cor:a-harm-repr} Let $u\in W^{1,2}(B_1,|y|^a)$, $-1<a<1$, be a weak
  solution of the equation $L_au=0$ in $B_1$. Then, we have a
  representation
  $$
  u(x,y)=\varphi(x,y)+y|y|^{-a}\psi(x,y),\quad (x,y)\in B_1,
  $$
  where $\varphi(x,y)$ and $\psi(x,y)$ are real analytic functions,
  even in $y$.
\end{corollary}

\begin{corollary}\label{cor:frac-analyt} Let
  $u\in\mathcal{L}_s(\R^n)$ satisfies $(-\Delta)^s u=0$ in the unit
  ball $B_1'\subset\R^n$. Then $u$ is real analytic in $B_1'$.
\end{corollary}

\begin{corollary}\label{cor:uniq-restr} Let $u\in W^{1,2}(B_1,|y|^a)$, $-1<a<1$, be a weak
  solution of the equation $L_au=0$ in $B_1$, even in $y$. If
  $u(\cdot,0)\equiv 0$ in $B_1'$, then $u\equiv 0$ in $B_1$.
\end{corollary}

The proof of Theorem~\ref{thm:even-a-harm} and subsequently those of
Corollaries~\ref{cor:a-harm-repr}, \ref{cor:frac-analyt}, and
\ref{cor:uniq-restr} are based on the following lemmas. We follow the
approach of \cite{AxlBouRam01} for harmonic functions.

\medskip Let $\mathcal{P}_m^*=\{p: \text{$p(x,y)$ polynomial of degree
  $\leq m$, even in $y$}\}$.

\begin{lemma} Let $p\in \mathcal{P}_m^*$. Then there exists
  $\tilde{p}\in \mathcal{P}_m^*$ such that
  $$
  L_a \tilde{p}=0\quad\text{in }B_1,\quad \tilde{p}=p\quad\text{on
  }\partial B_1.
  $$
  In other words, the solution of the Dirichlet problem for $L_a$ in
  $B_1$ with boundary values in $\mathcal{P}_m^*$ on $\partial B_1$ is
  itself in $\mathcal{P}_m^*$.

\end{lemma}
\begin{proof} For $m=0,1$, we simply have $\tilde{p}=p$. For $m\geq2$,
  we proceed as follows.

  For $q\in \mathcal{P}_{m-2}^*$ define $Tq\in \mathcal{P}_{m-2}^*$ by
$$
(Tq)(x,y)=|y|^{-a}L_a((1-x^2-y^2)q(x,y)).
$$
(It is straightforward to verify that $Tq$ is indeed in
$\mathcal{P}_{m-2}^*$). We now claim that the mapping
$T:\mathcal{P}_{m-2}^*\to \mathcal{P}_{m-2}^*$ is bijective. Since $T$
is clearly linear and $\mathcal{P}_{m-2}^*$ is finite dimensional it
is equivalent to showing that $T$ is injective. To this end, suppose
that $Tq=0$ for some $q\in \mathcal{P}_{m-2}^*$. This means that
$Q(x,y)=(1-x^2-y^2)q(x,y)$ is $L_a$-harmonic in $B_1$:
$$
L_a Q=0\quad\text{in }B_1.
$$
On the other hand $Q=0$ on $\partial B_1$ and therefore, by the
maximum principle $Q=0$ in $B_1$. But this implies that $q=0$ in
$B_1$, or that $q\equiv 0$. Hence, the mapping $T$ is injective, and
consequently bijective. It is now easy to see that
$$
\tilde{p}=p-(1-x^2-y^2)T^{-1}(|y|^{-a}L_a(p))\in \mathcal{P}_m^*
$$
satisfies the required properties.
\end{proof}
\begin{lemma} Polynomials, even in $y$, are dense in the subspace of
  functions in $L^2(\partial B_1,|y|^a)$, even in $y$.
\end{lemma}
\begin{proof} Polynomials, even in $y$ are dense in the space of
  continuous functions in $C(\partial B_1)$, even in $y$, with the
  uniform norm. The claim now follows from the observation that the
  embedding $C(\partial B_1)\hookrightarrow L^2(\partial B_1, |y|^a)$
  is continuous:
  \[
    \|v\|_{L^2(\partial B_1,|y|^a)}\leq \|v\|_{L^\infty(\partial
      B_1)}\left(\int_{\partial B_1}|y|^a \right)^{1/2}\leq C
    \|v\|_{L^\infty(\partial B_1)}.  \qedhere
  \]
\end{proof}

\begin{lemma}\label{lem:ortho-bas} The subspace of functions in $L^2(\partial B_1,|y|^a)$,
  even in $y$, has an orthonormal basis $\{p_{k}\}_{k=0}^\infty$
  consisting of homogeneous $L_a$-harmonic polynomials $p_k$, even in
  $y$.
\end{lemma}
\begin{proof} If $p$ is a polynomial, even in $y$, then restricted to
  $\partial B_1$ it can be replaced with an $L_a$-harmonic polynomial
  $\tilde p$. On the other hand, if we decompose
  $$
  \tilde{p}=\sum_{i=0}^m q_i
  $$
  where $q_i$ is a homogeneous polynomial of order $i$, even in $y$,
  then
$$
|y|^{-a}L_a \tilde{p}= \sum_{i=2}^m |y|^{-a}L_aq_i
$$
where $|y|^{-a}L_aq_i$ is a homogeneous polynomial of order $i-2$,
$i=2,\ldots,m$. Hence, $L_a \tilde{p}=0$ iff $L_a q_i=0$, for all
$i=0,\ldots, m$ (for $i=0,1$ this holds automatically).

We further note that if $q_i$ and $q_j$ are two homogeneous
$L_a$-harmonic polynomials of degrees $i\neq j$, then they are
orthogonal in $L^2(\partial B_1,|y|^a)$. Indeed,
\begin{align*}
  0=\int_{B_1}q_i\div(|y|^a\nabla q_j)-\div(|y|^a\nabla q_i)
  q_j&=\int_{\partial B_1} (q_i\partial_{\nu}q_j -q_j\partial_{\nu}
       q_i)|y|^a\\
     &=(j-i)\int_{\partial B_1}q_iq_j|y|^a.
\end{align*}
Using this and following the standard orthogonalization process, we
can construct a basis consisting of homogeneous $L_a$-harmonic
polynomials.
\end{proof}

\begin{lemma}\label{lem:Linf-L2bdry} Let $u\in W^{1,2}(B_1, |y|^a)\cap C(\overline{B_1})$ is a weak
  solution of $L_a u=0$ in $B_1$. Then
  $$
  \|u\|_{L^\infty(K)}\leq C_{n,a,K}\|u\|_{L^2(\partial B_1,|y|^a)}.
$$
for any $K\Subset B_1$.
\end{lemma}
\begin{proof} First, we note that by \cite{FraSer87}
  $$
  \|u\|_{L^\infty(K)}\leq C_{n,a,K}\|u\|_{L^2(B_1,|y|^a)}.
  $$
  So we just need to show that
  $$
  \|u\|_{L^2(B_1,|y|^a)}\leq C_{n,a}\|u\|_{L^2(\partial B_1,|y|^a)}.
$$
This follows from the fact that $u^2$ is a subsolution:
$L_a (u^2)\geq 0$ in $B_1$ and therefore the weighted spherical
averages
$$
r\mapsto\frac{1}{\omega_{n,a}r^{n+a}}\int_{\partial B_r} u^2
|y|^a,\quad 0<r<1
$$
are increasing. Integrating, we easily obtain that
\[
  \|u\|_{L^2(B_1,|y|^a)}\leq C_{n,a}\|u\|_{L^2(\partial B_1,|y|^a)}.
  \qedhere
\]
\end{proof}

We are now ready to prove Theorem~\ref{thm:even-a-harm}.

\begin{proof}[Proof of Theorem~\ref{thm:even-a-harm}] Without loss of
  generality we may assume
  $u\in W^{1,2}(B_1, |y|^a)\cap C(\overline{B_1})$, otherwise we can
  consider a slightly smaller ball. Now, using the orthonormal basis
  $\{p_k\}_{k=0}^\infty$ from Lemma~\ref{lem:ortho-bas} we represent
  $$
  u=\sum_{k=0}^\infty a_k p_k\quad\text{in }L^2(\partial B_1,|y|^a).
$$
We then claim that
$$
u(x,y)=\sum_{k=0}^\infty a_k p_k(x,y)\quad\text{uniformly on any
}K\Subset B_1.
$$
Indeed, if $u_m(x,y)=\sum_{k=0}^m a_k p_k(x,y)$, then
$\|u-u_m\|_{L^2(\partial B_1,|y|^2)}\to 0$ as $m\to \infty$ and
therefore by Lemma~\ref{lem:Linf-L2bdry}
\[
  \|u-u_m\|_{L^\infty(K)}\leq C_{n,a,K}\|u-u_m\|_{L^2(\partial
    B_1,|y|^a)}\to 0.\qedhere
\]
\end{proof}

We now give the proofs of the corollaries.

\begin{proof}[Proof of Corollary~\ref{cor:a-harm-repr}]
  Write $u(x,y)$ in the form
$$
u(x,y)=u_{\mathrm{even}}(x,y)+u_{\mathrm{odd}}(x,y),
$$
where $u_{\mathrm{even}}$ and $u_{\mathrm{odd}}$ are even and odd in
$y$, respectively. Clearly, both functions are
$L_a$-harmonic. Moreover, by Theorem~\ref{thm:even-a-harm},
$u_{\mathrm{even}}$ is real analytic and we take
$\varphi=u_{\mathrm{even}}$. On the other hand, consider
$$
v(x,y)=|y|^a\partial_y u_{\mathrm{odd}}(x,y).
$$
Then, $v$ is $L_{-a}$-harmonic in $B_1$ and again by
Theorem~\ref{thm:even-a-harm}, $v$ is real analytic. We can now
represent
$$
u_{\mathrm{odd}}(x,y)=y|y|^{-a}\psi(x,y),\quad
\psi(x,y)=y^{-1}|y|^{a}\int_0^y|s|^{-a} v(x,s)ds.
$$
It is not hard to see that $\psi(x,y)$ is real analytic, which
completes our proof.
\end{proof}
\begin{proof}[Proof of Corollary~\ref{cor:frac-analyt}] The proof
  follows immediately from Theorem~\ref{thm:even-a-harm} by
  considering the Caffarelli-Silvestre extension
  $$
  u(x,y)=u*P(\cdot,y)=\int_{\R^n} P(x-z,y)u(z)dz,\quad (x,y)\in
  \R^n\times\R_+
$$
where $P(x,y)=C_{n,a}\frac{y^{1-a}}{(|x|^2+y^2)^{(n+1-a)/2}}$ is the
Poisson kernel for $L_a$, and noting that its extension to $\R^{n+1}$
by even symmetry in $y$ (still denoted $u$) satisfies $L_au=0$ in
$B_1$.
\end{proof}
\begin{proof}[Proof of Corollary~\ref{cor:uniq-restr}]
  Represent $u(x,y)$ as a locally uniformly convergent in $B_1$ series
  $$
  u(x,y)=\sum_{k=0}^\infty q_k(x,y),
$$
where $q_k(x,y)$ is a homogeneous of degree $k$ $L_a$-harmonic
polynomial, even in $y$.  We have
$$
u(x,0)=\sum_{k=0}^\infty q_k(x,0)\equiv 0
$$
from which we conclude that $q_k(x,0)\equiv 0$. We now want to show
that $q_k\equiv 0$. To this end represent
$$
q_k(x)=\sum_{j=0}^{[k/2]} p_{k-2j}(x)y^{2j},
$$
where $p_{k-2j}(x)$ is a homogeneous polynomial of order $k-2j$ in
$x$. Clearly $p_{k}(x)\equiv0$. Taking partial derivatives
$\partial^{\alpha}_x q_k(x)$ of order $|\alpha|=k-2$, we see that
$$
\partial^{\alpha}_x q_k(x)=c_\alpha y^2,\quad
c_\alpha=\partial^{\alpha}_xp_{k-2}
$$
is $L_a$-harmonic, which can happen only when $c_\alpha=0$. Hence
$D^{k-2}_x p_{k-2}(x)\equiv 0$ and therefore $p_{k-2}\equiv 0$. Then
taking consequently derivatives of orders $k-2j$, $j=2,\ldots$, we
conclude that $p_{k-2j}(x)\equiv 0$ for all $j=0,\ldots,[k/2]$ and
hence $q_k(x,y)\equiv 0$.
\end{proof}


\end{document}